\documentclass[leqno]{amsart}
\usepackage{amsfonts,amssymb,amsmath,amsgen,amsthm}
\usepackage{enumitem}
\usepackage{hyperref}
\usepackage{color}
\newcommand{\msc}[2][2000]{%
  \let\@oldtitle\@title%
  \gdef\@title{\@oldtitle\footnotetext{#1 \emph{Mathematics subject
        classification.} #2}}% 
}

\theoremstyle{plain}
\newtheorem{theorem}{Theorem} [section]
\newtheorem{definition}[theorem]{Definition}

\newtheorem{lemma}[theorem]{Lemma}

\newtheorem{proposition}[theorem]{Proposition}
\newtheorem{hyp}[theorem]{Assumption}
\theoremstyle{remark}
\newtheorem{remark}[theorem]{Remark}

\def\R{{\mathbb R}}% real numbers
\def\N{{\mathbb N}}% nonnegative integers
\def\Sch{{\mathcal S}}% Schwartz space
\def\O{\mathcal O}

\def\dd{\mathrm d}

\def\({\left(}
\def\){\right)}
\def\<{\left\langle}
\def\>{\right\rangle}
\def\le{\leqslant}
\def\ge{\geqslant}

\def\Eq#1#2{\mathop{\sim}\limits_{#1\rightarrow#2}}
\def\Tend#1#2{\mathop{\longrightarrow}\limits_{#1\rightarrow#2}}

\def\d{{\partial}}
\def\eps{\varepsilon}

\def\si{{\sigma}}

\DeclareMathOperator{\DIV}{div}

\numberwithin{equation}{section}

\begin{document}

\title[Rigidity in isothermal fluids]{Rigidity results in
  generalized isothermal fluids}

\author[R. Carles]{R\'emi Carles}
\author[K. Carrapatoso]{Kleber Carrapatoso}
\author[M. Hillairet]{Matthieu Hillairet}

\address{Institut Montpelli\'erain
  Alexander Grothendieck\\ CNRS\\ Univ. Montpellier\\France}
\email{Remi.Carles@math.cnrs.fr}
\email{Kleber.Carrapatoso@umontpellier.fr}
\email{Matthieu.Hillairet@umontpellier.fr}
\begin{abstract}
  We investigate the long-time behavior of solutions to 
  the isothermal Euler, Korteweg or quantum Navier Stokes equations,
  as well as generalizations of these equations where the convex
  pressure law is asymptotically linear near vacuum. By writing
  the system with a 
  suitable time-dependent scaling 
  we prove that the densities of global solutions display universal
  dispersion rate and asymptotic profile. This result applies to  weak
  solutions  defined in an appropriate way. In the exactly isothermal
  case, we establish
   the compactness of bounded sets of  such weak solutions, by
   introducing modified entropies adapted to the new unknown
   functions. 
\end{abstract}

\thanks{This work was partially supported by the EFI project ANR-17-CE40-0030 and the Dyficolti project ANR-13-BS01-0003-01 of the French National Research Agency (ANR)}

\maketitle

%%%%%%%%%%%%%%%%%%%%%%%%%%%%%%%%%%%%%%%%%%%%%%

\section{Introduction}
\label{sec:intro}

In the isentropic case $\gamma>1$, the Euler equation on $\R^d$, $d\ge 1$,
\begin{equation}\label{eq:euler-isent}
  \left\{
    \begin{aligned}
 & \d_t \rho+\DIV\(\rho u\)=0,\\
     &\d_t (\rho u) +\DIV(\rho u\otimes u)+\nabla
        \(\rho^\gamma\)= 0,
    \end{aligned}
\right.
\end{equation}
enjoys the formal conservations of mass,
\begin{equation*}
  M(t) =\int_{\R^d}\rho(t,x)dx\equiv M(0),
\end{equation*}
and entropy (or energy),
\begin{equation*}
  E(t) = \frac{1}{2}\int_{\R^d} \rho (t,x)
  |u(t,x)|^2 \dd x+\frac{1}{\gamma-1}\int_{\R^d}\rho(t,x)^\gamma
  \dd x  \equiv E(0). 
\end{equation*}
In general, smooth solutions are defined only locally in time (see
\cite{MUK86,JYC90,Xin98}). However, for some range of $\gamma$, if the
initial velocity has a 
special structure and the initial density is sufficiently small, the
classical solution is defined globally in time. In addition
the large time behavior of the solution can be described rather
precisely, as established in    \cite{Serre97}. We restate some
results from \cite{Serre97} in the following theorem:
\begin{theorem}[From \cite{Serre97}]\label{theo:serre}
  Let $1<\gamma\le 1+2/d$ and $s>d/2+1$. There exists $\eta>0$ such
  that the following holds.\\
$(i)$ If $\rho_0,u_0\in H^s(\R^d)$ are such that
  $\|(\rho_0^{(\gamma-1)/2},u_0)\|_{H^s(\R^d)}\le \eta $, 
  then the system \eqref{eq:euler-isent} with initial data $\rho(0,x)
  = \rho_0(x)$ and $u(0,x) = x + u_0(x)$ admits a unique global
  solution, in the sense that $(\rho,\tilde u)\in
      C([0,\infty);H^s(\R^d))$, where $\tilde
      u(t,x)=u(t,x)-\frac{x}{1+t}$. In addition, there exists $R_\infty,U_\infty\in
  H^s(\R^d)$ such that
 \begin{equation}\label{eq:tempsgrandSerre}
    \left\| \(\rho (t,x)-\frac{1}{t^d}R_\infty\(\frac{x}{t}\),u(t,x)
      -\frac{x}{1+t}-\frac{1}{1+t}U_\infty
      \(\frac{x}{t}\)\)\right\|_{L^\infty(\R^d)}\Tend t \infty 0. 
  \end{equation}
$(ii)$ Conversely, if $R_\infty,U_\infty\in H^s(\R^d)$ are such
  that 
  $\|(R_\infty^{(\gamma-1)/2},U_\infty)\|_{H^s(\R^d)}\le \eta$, then
there exists $\rho_0,u_0\in H^s(\R^d)$ such that the solution to
\eqref{eq:euler-isent} with $ \rho(0,x)=\rho_0(x)$ and 
 $ u(0,x) =    x+u_0(x)$ is global in time in the same sense as above,
 and \eqref{eq:tempsgrandSerre} holds. 
\end{theorem}
In particular, in the frame of small data (in the sense described above), 
the dispersion  
\begin{equation*}
  \left\| \rho(t)\right\|_{L^\infty(\R^d)}\Eq t \infty \frac{\|R_\infty\|_{L^\infty(\R^d)}}{t^d}
\end{equation*}
is universal but the asymptotic profile $R_\infty$ can be arbitrary. Typically,
given any function $\psi\in \Sch(\R^d)$, $R_\infty= \epsilon \psi$ will
be allowed provided that $\epsilon>0$ is sufficiently small. For completeness
we provide a brief proof of the above theorem in appendix. 
\smallbreak 

We emphasize that the structure of the velocity is crucial: the
initial velocity is a small (decaying) perturbation of a linear
velocity. In a way, the above result is the Euler generalization of
the global existence results for the Burgers equation with expanding
data. Refinements of this result can be found in
\cite{GrSe97,Gra98,Serre2017}.  
\smallbreak

%\subsection{Isothermal case and generalizations}

The isothermal Euler equation corresponds to the value $\gamma=1$ in
\eqref{eq:euler-isent}, 
\begin{equation} \label{eq:euler-isotherme}
  \left\{
    \begin{aligned}
 & \d_t \rho+\DIV\(\rho u\)=0,\\
     &\d_t (\rho u) +\DIV(\rho u \otimes u)+\kappa\nabla
        \rho= 0,\quad \kappa>0,
    \end{aligned}
\right.
\end{equation}
The mass is still formally conserved, and the energy now reads
\begin{equation*}
  E(t) = \frac{1}{2}\int_{\R^d} \rho (t,x)
  |u(t,x)|^2 \dd x+\int_{\R^d}\rho(t,x)\ln \rho(t,x)
   \dd x\equiv E(0). 
\end{equation*}
Unlike in the isentropic case, the energy has an indefinite sign, a
property which causes many technical problems. %In this article, we
%introduce a change of unknown functions that removes this property:
%a Lyapunov functional is naturally associated to these new unknowns,
%which is the sum of non-negative terms. 
In this paper, we show that the isothermal Euler equation on $\R^d$,
$d\ge 1$, with asymptotically vanishing density, $\rho(t,\cdot)\in L^1(\R^d)$, 
displays a specific large time behavior, in the sense that if the
solution is global in time, then  the density disperses with a  rate different from the above one, and possesses a universal asymptotic Gaussian profile. This property remains when the convex pressure law $P(\rho)$ satisfies $P'(0)>0$, as well as  for the
Korteweg and quantum Navier-Stokes equations:
\begin{equation}\label{eq:NSK}
  \left\{
    \begin{aligned}
   & \d_t \rho+\DIV\(\rho u\)=0,\\
   &\d_t (\rho u) +\DIV (\rho u \otimes u) +
      \nabla P(\rho)= \frac{\eps^2}{2}\rho\nabla\( \frac{\Delta
        \sqrt{\rho}}{\sqrt{\rho}}\) + \nu \DIV \(\rho D u\),
    \end{aligned}
\right.
\end{equation}
with $\eps,\nu\ge 0$, where $Du$ denotes the symmetric part of the gradient,
\begin{equation*}
  Du :=  \frac{1}{2}\(\nabla u +\,^t\nabla u\).
\end{equation*}
For this system, we still have conservation of mass and the energy
 \begin{equation}\label{eq:energy-init}
    E(t) = \frac{1}{2}\int \rho(t,x)|u(t,x)|^2 \dd x+\frac{\eps^2}{2} \int |\nabla
    \sqrt{\rho(t,x)}|^2 \dd x+\int F(\rho(t,x)) \dd x,
  \end{equation}
where
\begin{equation*}
  F(\rho) = \rho\int_1^\rho \frac{P(r)}{r^2} \dd r,
\end{equation*}
satisfies
\begin{equation*}
  \dot E(t) =- \nu \int \rho |Du|^2.
\end{equation*}

In the case $\varepsilon =0$ and  $P(\rho)= \kappa \rho,$ equation
\eqref{eq:NSK}  is the precise system derived in \cite{Bru-Me10}, as a
correction to the isothermal quantum Euler equation. We emphasize
that, because of the lack of positivity of the term $F$ in the energy
functional, only the barotropic variant -- where $P(\rho) = \kappa
\rho^{\gamma}$ with $\gamma >1$ --  is studied in references.
Classically, a Bohm potential (corresponding to the term multiplied by
$\varepsilon$ in \eqref{eq:NSK}) is also added, see
\cite{AS-comp,Gis-VV15,Jungel,VasseurYu} for instance. In the case
where the dissipation is absent ($\nu=0$), but with capillarity
($\eps>0$), we refer to \cite{BDD07,AnMa09,AnMa12,AuHa17}. 
\smallbreak

A loose statement of our main result reads (a more precise version is
provided in the  next section, see Theorem~\ref{theo:temps-long}): 
\begin{theorem} \label{thm_mainrough}
Let $(\rho,u)$ be a global weak solution to \eqref{eq:NSK} with initial density/velocity 
$(\rho_0,u_0)$ satisfying
\[
(1+|x|^2 +|u_0|^2)^{1/2} \sqrt{\rho_0}\in L^2(\mathbb R^d) .
\]
Then there exists a mapping $\tau : [0,\infty) \to [1,\infty)$ such that
\begin{align*}
& \tau(t) \Eq t \infty 2t \sqrt{P'(0) \ln(t)}, \\
&  \rho(t, x) \Eq t \infty \dfrac{\|\rho_0\|_{L^1}}{\pi^{d/2}}
  \dfrac{\exp(-|x|^2/\tau(t)^2)}{\tau(t)^{d}}. 
\end{align*}
\end{theorem}

This theorem entails that, in contrast with the isentropic case, the
density of solutions to \eqref{eq:NSK} disperses as follows : 
\begin{equation*}
  \left\| \rho(t)\right\|_{L^\infty(\R^d)}\Eq t \infty
  \frac{\|\rho_0\|_{L^1(\R^d)}}{\(2P'(0)\sqrt{\pi}\)^d}\times 
\frac{1}{\(t\sqrt{\ln(t)}\)^d}, 
\end{equation*}
with a universal profile. This result applies to a notion of ``weak
solution'' that is based 
on standard {\em a priori} estimates satisfied by smooth solutions to \eqref{eq:NSK}.
We make precise the definition of such solutions in the next section, see Definition~\ref{def:weaksol-minus}. 
\smallbreak

 The main ingredient of the proof is to translate in terms of our
 isothermal equations a change of unknown functions introduced for the
 dispersive logarithmic Schr\"odinger equation  in \cite{CaGa-p}. This
 enables to transform \eqref{eq:NSK} into a system with unknowns $(R,U)$
 for which the associated energy is positive-definite. A second
 feature of the new system is that, asymptotically in time, it reads
 (keeping only the dominating terms): 
\begin{equation}\label{eq:NSK-asymptotic}
  \left\{
    \begin{aligned}
   & \d_t R + \dfrac{1}{\tau^2}\DIV\(R U\)=0,\\
   &\d_t (R U)  +  2 P'(0) y R  + P'(0) \nabla R = 0,  
    \end{aligned}
\right.
\end{equation}
where $\tau$ is the time-dependent scaling mentioned in
Theorem~\ref{thm_mainrough}.  
By taking the divergence of the second equation and replace ${\d_t \rm div(RU)}$ with the first one, we obtain then (keeping again only the dominating terms):
\[
\left\{
\begin{aligned}
& \d_t R  = 0,\\
& \d_t R - P'(0) \mathcal L R = 0, 
\end{aligned}
\right.
\]
where $\mathcal L$ is the Fokker-Planck operator $\mathcal L R = \Delta R + 2 {\rm div}(y R).$ 
In this last system, the first equation implies that $R$ converges to a stationary solution to the second equation. The analysis of the long-time behavior of solutions to this Fokker-Planck equation, as provided in \cite{AMTU01}, entails the expected result. 
\smallbreak 

The outline of the paper is as follows. In the next section, we
provide rigorous definitions of weak solutions and precise statements
for our main result.  Section~\ref{sec:longtime} is then devoted to
the long-time behavior of solutions 
to \eqref{eq:NSK}. In this section, we compute at first explicit
solutions to \eqref{eq:NSK} with Gaussian densities.  These explicit
computations motivate the introduction of the change of variable that
we use afterwards. In what remains of this 
section we give an exhaustive proof of the precise version for
Theorem~\ref{thm_mainrough}.   
The long-time analysis mentioned here is based on the {\em a priori}
existence of solutions. 
However, in the compressible setting, global existence of solutions is
questionable. So, in the last section of the paper, we focus on the
notion of weak solutions that we consider. At first, we present the
{\em a priori} estimates which motivate their definition. We end the
paper by proving a sequential compactness result. This sequential
compactness property is a cornerstone for the proof of existence of
weak solutions, see e.g. \cite{Lio98,Fei04}. As for the large time
behavior, we simply state a loose version of our result here (see Theorem~\ref{theo:compacite} for the precise statement):
\begin{theorem}
 Assume $\nu >0$, $0\le \eps \le \nu$, $P(\rho)=\kappa \rho$ with
 $\kappa>0$, and let 
 $T>0$. Let $(\rho_n,u_n)_{n\in \N}$ be a sequence of weak solutions to
 \eqref{eq:NSK} on $(0,T)$, enjoying uniformly the conservation of
 mass, as well as a suitable 
 notion of energy dissipation, BD-entropy dissipation, and
 Mellet-Vasseur type inequality. Then up to the extraction of a
 subsequence, $(\rho_n,u_n)_{n\in \N}$ converges to a weak solution of \eqref{eq:NSK}
 on $(0,T)$.
\end{theorem}
It is for
the system \eqref{eq:RU} in terms of $(R,U)$, as mentioned above, that fairly natural a priori estimates are required in the
above statement. Even though the notions of solution for
\eqref{eq:NSK} and \eqref{eq:RU} are equivalent
(Lemma~\ref{lem:equivalence} below), we did not find a direct approach
to express the pseudo energy, pseudo BD-entropy and Mellet-Vasseur
type inequality mentioned above in a direct way in terms of
$(\rho,u)$, that is, without resorting to $(R,U)$.

%%%%%%%%%%%%%%%%%%%%%%%%%%%%%%%%%%%%%%%%%%%%%%

\section{Weak solutions and large time behavior}
\label{sec:main}

We now state a precise definition regarding the notion of solution
that we consider in this paper. 
Even though, in \eqref{eq:NSK}, the fluid genuine unknowns are $\rho$
and $u$, the mathematical theory that we develop  in
Section~\ref{sec:compactness} suits 
better to the unknowns $\sqrt{\rho}$ and $\sqrt{\rho}u$. Therefore we state
our definition of weak solution 
in terms of these latter unknowns. Nevertheless, we shall keep these notations, even though no fluid velocity field $u$ underlies the computation of $\sqrt{\rho}u.$ 

\begin{definition} \label{def:weaksol-minus}
Let $\nu \ge 0$ and $\varepsilon \ge 0.$ Given $T >0$, we call weak solution to \eqref{eq:NSK} on $(0,T)$ any pair $(\rho,u)$ such that there is a collection $(\sqrt{\rho},\sqrt{\rho}u ,\mathbf
S_{K},\mathbf T_{N})$ satisfying
\begin{itemize}
\item[i)] The  following regularities:
\begin{align*}
&\(\<x\> +|u|\)\sqrt{\rho} \in L^{\infty}\(0,T; L^2 (\R^d)\),\quad \text{where
  }\<x\> =\sqrt{1+|x|^2}, \\
&(\eps+\nu)\nabla \sqrt \rho \in L^{\infty}\(0,T; L^2 (\R^d)\),\\
&\varepsilon \sqrt \nu \,  \nabla^2 \sqrt{\rho} \in L^2(0,T;L^2(\mathbb
  R^d)) ,\\
& \sqrt \nu \, \mathbf T_N \in L^2(0,T; L^2(\mathbb R^d)) ,
\end{align*}
with the compatibility conditions
\[
\sqrt{\rho} \ge 0 \text{ a.e.  on } (0,T)\times \R^d,  \quad   \sqrt{\rho}u=
0 \text{ a.e. on } \left\{\sqrt{\rho} = 0\right\}.
\]

\item[ii)] {\normalfont \bf Euler case $\eps=\nu=0$:} The following equations in $\mathcal D'((0,T)\times \mathbb R^d)$
\begin{equation}\label{eq:NSKminus-bis-euler}
  \left\{
    \begin{aligned}
  & \d_t{\rho}+\DIV ({\rho} u )=  0,\\
    &\d_t ({\rho}u) +\DIV ( \sqrt{\rho}u \otimes \sqrt{\rho}u)
      +  \nabla P(\rho) 
     =0 .
    \end{aligned}
\right.
\end{equation}

\item[iii)]{\normalfont \bf Korteweg and Navier-Stokes cases $\eps + \nu > 0$:} The following equations in $\mathcal D'((0,T)\times \mathbb R^d)$
\begin{equation}\label{eq:NSKminus-bis}
  \left\{
    \begin{aligned}
  & \d_t\sqrt{\rho}+\DIV (\sqrt{\rho} u )=  \frac{1}{2}{\rm Trace}(\mathbf T_N),\\
    &\d_t ({\rho}u) +\DIV ( \sqrt{\rho}u \otimes \sqrt{\rho}u)
      +  \nabla P(\rho) 
     =\DIV \left(\nu\sqrt \rho \mathbf S_N
        + \dfrac{\varepsilon^2}{2} \mathbf S_K\right) ,
    \end{aligned}
\right.
\end{equation}
with $\mathbf S_N$ the symmetric part of $\mathbf T_N$, and the
compatibility conditions: 
\begin{align} \label{eq_compnewton-minus}
& \sqrt{\rho}\mathbf T_{N} = \nabla(\sqrt{\rho}  \sqrt{\rho} u) - 2
  \sqrt{\rho}u \otimes \nabla \sqrt{\rho}\,, \\[6pt] 
& \mathbf S_K 
=\sqrt{\rho}\nabla^2 \sqrt{\rho} -  \nabla \sqrt{\rho} \otimes \nabla
  \sqrt{\rho} . \label{eq_compbohm-minus}
\end{align}
\end{itemize}
\end{definition}

We emphasize that the above definition is essentially the ``standard''
one, up to the fact that we require $|x|\sqrt\rho\in
L^\infty(0,T;L^2(\R^d))$. The reason for this assumption will become
clear in the Subsection~\ref{sec:apriori} where we will recall the {\em a
  priori estimates} motivating this definition (see
Lemma~\ref{lem:equivalence}, as well as the definition of the
pseudo-energy $\mathcal E$ in \eqref{eq:pseudo-energy}). 
\smallbreak

Several remarks are in order. When the symbol $\rho$ alone appears, it 
must be understood as $|\sqrt{\rho}|^2$, while when the symbol $u$
appears alone, it is defined by $u =
\sqrt{\rho}u/\sqrt{\rho}\,\mathbf 1_{\sqrt{\rho}>0}.$ Under the compatibility
condition of item $i)$ this yields a well-defined vector-field.  
As for the stress-tensors involved in the momentum equation
\eqref{eq:NSKminus-bis}, we
emphasize that \eqref{eq_compnewton-minus} reads formally $\mathbf T_N = \sqrt{\rho} \nabla u.$ 
\smallbreak

An originality of the previous definition is that in the case
$\eps+\nu>0$, we do not ask for the continuity equation in terms of
$\rho$ but in terms of $\sqrt{\rho}.$ However, we prove here that 
the usual continuity equation as written in \eqref{eq:NSKminus-bis} is
a consequence to 
this definition thanks to the regularity of $\sqrt{\rho}$ and
$\sqrt{\rho}u.$ This is the content of the following lemma: 
\begin{lemma}\label{lem:continuite}
Let $\eps + \nu >0$.
Assume that $(\rho,u)$ is a weak solution to \eqref{eq:NSKminus-bis} on
$(0,T)$  in the sense of Definition~\ref{def:weaksol-minus}. 
Then it satisfies
\[
\partial_t \rho + {\rm div}(\rho u) = 0 \quad \text{ in
  $\mathcal D'((0,T) \times \mathbb R^d).$} 
\]
\end{lemma}
\begin{proof}
By definition, we have 
\[
\d_t\sqrt{\rho}+\DIV\(\sqrt{\rho} u\)=  \frac{1}{2}{\rm Trace}(\mathbf T_N)
\]
Here we note that $\sqrt{\rho}u \in
L^{\infty}(0,T;L^{2}(\mathbb R^d))$ (so that ${\rm div}(\sqrt{\rho}u) \in
L^{\infty}(0,T;H^{-1}(\mathbb R^d))$).  We can then multiply this
equation by $\sqrt{\rho} \in L^{\infty}(0,T;H^1(\mathbb R^d)).$
We
obtain: 
\[
\partial_t {\rho}  = - 2\sqrt{\rho}\, {\rm div}(\sqrt{\rho}u)
+ \sqrt{\rho}\, {\rm Trace}(\mathbf T_N).
\]
At this point we remark that, by definition of $\mathbf T_N:$
\[
{\rm div}(\rho u) = \sqrt{\rho}\, {\rm Trace}(\mathbf T_N) + 2 \sqrt{\rho} u
\cdot \nabla \sqrt{\rho}   
\]
and, since $\rho u = \sqrt{\rho} \sqrt{\rho} u$, the products of
the identity below are well-defined: 
\[
{\rm div}(\rho u) = \sqrt{\rho}\, {\rm div}(\sqrt{\rho}u) +
\sqrt{\rho} u \cdot \nabla \sqrt{\rho}. 
\]
Combining these equation entails 
\[
{\rm div}(\rho u) = 2\sqrt{\rho}\,  {\rm div}(\sqrt{\rho} u)  -
\sqrt{\rho}\, {\rm Trace}{\mathbf T}_N. 
\]
We conclude thus that:
\[
\partial_t {\rho}  = -{\rm div}(\rho u).
\]
\end{proof}

\subsection{Rewriting of \eqref{eq:NSK} with a suitable time-dependent scaling}
\label{sec:large-time}
In the case where the density $\rho$
is defined for all time and is dispersive (in the sense that it goes
to zero pointwise), it is natural to examine
the behavior of $P$ near $0$, since it gives an ``asymptotic
pressure law'' as time goes to infinity. A consequence of our
result is that the large time behavior in \eqref{eq:NSK} is very
different according to $P'(0)>0$ or $P'(0)=0$. 
Herein, we assume that $P\in
C^2(0,\infty;\R^+)$ with $P'(0)>0$ and $P''\ge 0$.
Typically, when $P''\equiv 0$, we recover the isothermal case,
$P(\rho)=\kappa\rho$, and we can also consider
\begin{equation*}
  P(\rho)=\kappa \rho +\sum_{j=1}^N
  \kappa_j\rho^{\gamma_j},\quad N\ge 1,\ \kappa_j>0,\ \gamma_j>1,
\end{equation*}
with no other restriction on $\gamma_j$ (in any dimension), or even the exotic case
$P(\rho)=e^\rho$. 
The most general class of pressure laws that we shall consider 
is fixed by the following assumptions:
\begin{hyp}[Pressure law]\label{hyp:P-temps-long}
  The pressure $P\in C^1(\R^+;\R^+)\cap C^2(0,\infty;\R^+)$ is convex
  ($P''(\rho)\ge 0$ for all 
  $\rho> 0$), and satisfies
\[ 
\kappa:=P'(0)>0.
\]
\end{hyp}

 Resuming the
approach from \cite{CaGa-p} (the link between Schr\"odinger equation
and Euler-Korteweg equation is formally given by the Madelung transform), we
change the unknown functions as follows. Introduce $\tau(t)$
solution of the ordinary
differential equation
\begin{equation}\label{eq:tau}
  \ddot \tau = \frac{2\kappa }{\tau} \, ,\quad \tau(0)=1\, ,\quad \dot
  \tau(0)=0\, .
\end{equation}
The reason for considering this equation will become clear in
Subsection~\ref{sec:explicit}.  We find in \cite{CaGa-p}, for slightly
more general initial data:
\begin{lemma}\label{lem:taubis}
 Let $\alpha,\kappa>0$, $\beta\in \R$. Consider the ordinary differential equation
\begin{equation}\label{eq:taugen}
  \ddot \tau = \frac{2\kappa }{\tau} \, ,\quad \tau(0)=\alpha\, ,\quad \dot
  \tau(0)=\beta .
\end{equation}
It has a unique solution $\tau\in C^2(0,\infty)$, and it satisfies, as
$t\to \infty$, 
\begin{equation*}
  \tau(t)= 2t \sqrt{\kappa \ln t}\(1+\O(\ell(t))\) \, ,\quad \dot
  \tau(t)=2\sqrt{\kappa\ln  t}\(1+\O(\ell(t))\),
\end{equation*} 
where
\begin{equation*}
\ell (t):= \frac{\ln \ln t}{ \ln t} \, \cdotp
\end{equation*}
\end{lemma}
We
sketch the proof of this lemma in Appendix~\ref{sec:ODE}, without
paying attention to the quantitative 
estimate of the remainder term. 
We now introduce the Gaussian $\Gamma(y)=e^{-|y|^2}$, and we set
\begin{equation}
  \label{eq:uvFluid}
  \rho(t,x) =
  \frac{1}{\tau(t)^d}R\(t,\frac{x}{\tau(t)}\)
\frac{\|\rho_0\|_{L^1}}{\|\Gamma\|_{L^1}},\quad 
  u(t,x) = \frac{1}{\tau(t)} U \(t,\frac{x}{\tau(t)}\) +\frac{\dot \tau(t)}{\tau(t)}x,
\end{equation}
where we denote by $y$ the
spatial variable for $R$ and $U$.
Denoting $\theta=\frac{\| \rho_0 \|_{L^1}}{\|\Gamma\|_{L^1}}$,
\eqref{eq:NSK} becomes, in terms of these new unknowns,
\begin{equation}\label{eq:RU}
  \left\{
    \begin{aligned}
  & \d_tR+\frac{1}{\tau^2}\DIV\(R U\)=0,\\
    &\d_t (R U) +\frac{1}{\tau^2}\DIV ( R U \otimes U)
      +2\kappa y R 
      +P'\( \tfrac{\theta R}{\tau^d}\) \nabla R  \\
&\phantom{\d_t (R w) +\frac{1}{\tau^2}\DIV   }
=\frac{\eps^2}{2\tau^2}R\nabla\( \frac{\Delta  \sqrt{R}}{\sqrt{R}}\)   
+\frac{\nu}{\tau^2} \DIV (R DU) + \frac{\nu \dot \tau}{\tau} \nabla R.
    \end{aligned}
\right.
\end{equation}
The analogue of Definition~\ref{def:weaksol-minus} is the following:
\begin{definition} \label{def_weaksol}
Let $\nu \ge 0$ and $\varepsilon \ge 0.$ Given $T >0$, we call weak
solution to \eqref{eq:RU} on $(0,T)$ any pair $(R,U)$ such that there exists a collection $(\sqrt{R},\sqrt{R}U,\mathbb S_{K},\mathbb
T_{N})$ satisfying 
\begin{itemize}
\item[i)] The  following regularities:
\begin{align*}
&\(\<y\>+|U|\) \sqrt{R} \in L^{\infty}\(0,T; L^2 (\R^d)\),\\
&(\eps+\nu)\nabla \sqrt R \in L^{\infty}\(0,T; L^2 (\R^d)\),\\
&\varepsilon \sqrt \nu \,  \nabla^2 \sqrt{R} \in L^2(0,T;L^2(\mathbb
  R^d)) ,\\
& \sqrt \nu \, \mathbb T_N \in L^2(0,T; L^2(\mathbb R^d)) ,
\end{align*}
with the compatibility conditions
\[
\sqrt{R} \ge 0 \text{ a.e.  on } (0,T)\times \R^d,  \quad   \sqrt{R}U=
0 \text{ a.e. on } \{\sqrt{R} = 0 \}.
\]

\item[ii)] {\normalfont \bf Euler case $\eps=\nu=0$:} The following equations in $\mathcal D'((0,T)\times \mathbb R^d)$
\begin{equation}\label{eq:NSKrevu-euler}
  \left\{
    \begin{aligned}
  & \d_t{R}+\frac{1}{\tau^2}\DIV ({R} U )=  0,\\
    &\d_t ({R}U) +\frac{1}{\tau^2}\DIV ( \sqrt{R}U \otimes \sqrt{R}U)
      +2\kappa y R + P'\(\tfrac{\theta R}{\tau^d}\)\nabla R =0.
    \end{aligned}
\right.
\end{equation}

\item[iii)]{\normalfont \bf Korteweg and Navier-Stokes cases $\eps + \nu > 0$:} The following equations in $\mathcal D'((0,T)\times \mathbb R^d)$
\begin{equation}\label{eq:NSKrevu}
  \left\{
    \begin{aligned}
  & \d_t\sqrt{R}+\frac{1}{\tau^2}\DIV (\sqrt{R} U )=  \frac{1}{2\tau^2}{\rm Trace}(\mathbb T_N),\\
    &\d_t ({R}U) +\frac{1}{\tau^2}\DIV ( \sqrt{R}U \otimes \sqrt{R}U)
      +2\kappa y R + P'\(\tfrac{\theta R}{\tau^d}\)\nabla R 
      \\
      & \phantom{\d_t ({R}U) +\frac{1}{\tau^2}\DIV }=\DIV \left(\dfrac{\nu}{\tau^2} \sqrt R \mathbb S_N
        + \dfrac{\varepsilon^2}{2\tau^2} \mathbb S_K\right) +
      \dfrac{\nu \dot{\tau}}{\tau} \nabla R,
    \end{aligned}
\right.
\end{equation}
with $\mathbb S_N$ the symmetric part of $\mathbb T_N$ and the compatibility conditions:
\begin{align} \label{eq_compnewton}
& \sqrt{R}\mathbb T_{N} = \nabla(\sqrt{R}  \sqrt{R} U) - 2 \sqrt{R}U \otimes \nabla \sqrt{R}\,, \\[6pt]
& \mathbb S_K 
=\sqrt{R}\nabla^2 \sqrt{R} -  \nabla \sqrt{R} \otimes \nabla \sqrt{R}
  \,, \label{eq_compkorteweg} 
\end{align}
\end{itemize}
\end{definition}
Mimicking the proof of Lemma~\ref{lem:continuite}, we see that in the
case $\eps + \nu >0$, if 
$(R,U)$ is a weak solution to \eqref{eq:NSKrevu} on
$(0,T)$  in the sense of Definition~\ref{def_weaksol}, then
it satisfies 
\[
\partial_t R +\frac{1}{\tau^2} {\rm div}(R U) = 0 \quad \text{ in
  $\mathcal D'((0,T) \times \mathbb R^d).$} 
\]
In view of \eqref{eq:uvFluid}, we check directly:
\begin{lemma}[Equivalence of the notions of solution]\label{lem:equivalence}
  Let $T>0$. Then $(\rho,u)$ is a weak solution of \eqref{eq:NSK} on
  $(0,T)$ if and only if  $(R,U)$ is a weak solution of \eqref{eq:RU} on
  $(0,T)$, where $(\rho,u)$ and $(R,U)$ are related through \eqref{eq:uvFluid}.
\end{lemma}
\begin{remark}
  If in Definition~\ref{def:weaksol-minus}, we had required only
  $(1+|u|)\sqrt\rho\in L^\infty(0,T; L^2(\R^d))$, then the above equivalence
  would not hold. In the same spirit, the change of unknown
  \eqref{eq:uvFluid}  would make the notion of solution rather
  delicate in the case of the Newtonian Navier-Stokes equation, a case
  where typically $u\in L^2(0,T; H^1(\R^d))$. More generally, we do not consider
  velocities enjoying integrability properties, unless the density
  appear as a weight in the integral. 
\end{remark}
\medskip

We define the pseudo-energy $\mathcal E$ of the system \eqref{eq:RU} by
\begin{equation}\label{eq:pseudo-energy}
  \begin{aligned}
    \mathcal E(t) &:= \frac{1}{2\tau^2}\int
  R|U|^2+\frac{\eps^2}{2\tau^2} \int|\nabla \sqrt R|^2 +P'(0)\int (R |y|^2
   + R \ln R) \\
&\quad+ \frac{\tau^d}{\theta}\int G\(\frac{\theta R}{\tau^d}\) ,
  \end{aligned}
  \end{equation}
where
\begin{equation*}
   G(u) =\int_0^u\int_0^v \frac{P'(\sigma)-P'(0)}{\sigma}  \dd \sigma  \dd v,
\end{equation*}
which formally satisfies 
\begin{equation}\label{eq:evol-pseudo}
   \dot{\mathcal E}(t) =- \mathcal D(t) - \nu \frac{\dot
     \tau(t)}{\tau(t)^3} \int R(t,y)  \DIV U(t,y)\dd y ,
\end{equation}
where the dissipation $\mathcal D(t) $ is defined by
\begin{equation}\label{eq:dissip-pseudo}
\begin{aligned}
\mathcal D(t) 
& := \frac{\dot \tau}{\tau^3}\int R |U|^2 +
d\frac{\dot   \tau}{\tau} \tau^d \(\int \left[ P\(\sigma\) -\sigma
  P'(0) \right] |_{\sigma= \frac{ \theta R}{\tau^d}} \) \\
&\quad 
+\eps^2 \frac{\dot \tau}{\tau^3} \int |\nabla \sqrt R|^2 
+  \frac{\nu}{\tau^4} \int |\mathbb S_N|^2.
\end{aligned} 
\end{equation}
By convexity we have $G\ge 0$, and $P(\sigma)\ge P'(0)\sigma$ for
$\sigma\ge 0$, so 
$\mathcal D(t)\ge 0$. Note also the identities 
\begin{equation}
  \label{eq:FG}
  F''(\sigma) = \frac{P'(0)}{\sigma}+G''(\sigma),\quad F(\rho ) =
  P'(0)\rho\ln \rho + G(\rho). 
\end{equation}
Recall the Csisz\'ar-Kullback inequality (see
e.g.\ \cite[Th.~8.2.7]{LogSob}): for $f,g\ge 0$ with $\int_{\R^d} f=\int_{\R^d} g$,
\begin{equation*}
  \|f-g\|_{L^1({\mathbb R}^d)}^2\le 2 \|f\|_{L^1({\mathbb R}^d)}\int f(x)\ln \left(\frac{f(x)}{g(x)}\right)  \dd x.
\end{equation*}
Writing
\[
\int (R |y|^2
   + R \ln R) = \int R \ln \frac{R}{\Gamma},
\]
the (formal) conservation of the mass for $R$ and the definition
\eqref{eq:uvFluid}  imply  that the
pseudo-energy $\mathcal E$ is non-negative, 
$\mathcal E\ge 0$. 
\smallbreak

 As for global solutions, we have the following natural definition:

\begin{definition}\label{def:sol-globale}
Let $\nu \ge 0$ and $\varepsilon \ge 0.$ We call global weak solution to \eqref{eq:RU} any pair $(R,U)$ which, by restriction, yields a weak solution to 
\eqref{eq:RU} on $(0,T)$ for arbitrary $T>0.$ 
\end{definition}

\smallbreak

\subsection{Main result: large-time behavior of weak solutions to \eqref{eq:NSKrevu}}

With the previous definitions and remarks, a quantitative and precise statement
of Theorem \ref{thm_mainrough} reads as follows:
\begin{theorem}\label{theo:temps-long}
  Let $\eps,\nu\ge 0$. Assume that $P$ satisfies
  Assumption~\ref{hyp:P-temps-long}, 
  and let $(R,U)$ be a global weak solution of \eqref{eq:RU}, in the sense
  of Definition~\ref{def:sol-globale}, with constant mass
  \begin{equation*}
    \int_{\R^d}R(t,y) \dd y =  \int_{\R^d}R(0,y) \dd y,\quad \forall t\ge 0. 
  \end{equation*}
  \begin{enumerate}[label={\normalfont(\alph*)}]
  
\item  If $\displaystyle
\sup_{t \ge 0} \mathcal E(t) < \infty,$ then 
\[
\int_{\R^d} yR(t,y)\dd y\Tend t \infty 0\quad \text{and}\quad \left|\int_{\R^d}
  (RU)(t,y)\dd y\right|\Tend t \infty \infty, 
\]
 unless $\int y R(0,y)\dd y= \int (RU)(0,y)\dd y =0$, a case where 
\[
\int_{\R^d} yR(t,y)\dd y=\int_{\R^d}
  (RU)(t,y)\dd y\equiv 0.
\]

\smallskip
  \item If $\displaystyle
\sup_{t \ge 0} \mathcal E(t) + \int_0^\infty \mathcal D(t) \, \dd t < \infty,$
then
$R(t, \cdot) \rightharpoonup \Gamma$ weakly in $L^1(\R^d)$ as $ t \to \infty$.

\smallskip
\item If $\displaystyle
\sup_{t \ge 0} \mathcal E(t) < \infty$ and  the energy $E$
  defined by \eqref{eq:energy-init} 
satisfies $E(t)=o\(\ln t\)$ as $t\to \infty$, then
\begin{equation*}
  \int_{\R^d}|y|^2R(t,y)  \dd y \Tend t \infty \int_{\R^d}|y|^2 \Gamma(y) \dd y.
\end{equation*}
  \end{enumerate}
\end{theorem}

\begin{remark}
  Unlike in Theorem~\ref{theo:serre}, no smallness assumption is made
  on $U$ at $t=0$ ($U$ may even be linear in space), so there is no
  such geometrical structure on the initial velocity as in
  \cite{Serre97,Gra98}. 
\end{remark}

\begin{remark}\label{rem:hyp-naturelles}
  In view of \eqref{eq:evol-pseudo}--\eqref{eq:dissip-pseudo} and the
    property $\mathcal E\ge 0$, the assumptions of point (b) are
    fairly natural, after noticing that
\begin{align*}
\frac{\dot
     \tau(t)}{\tau(t)^3} \int R \left| \DIV U\right|\dd y & \le
   \frac{\dot \tau}{\tau}\(\int R\dd y\)^{1/2} \(\frac{1}{\tau^4}\int
   R|\DIV U|^2\dd y\)^{1/2}\\
&\lesssim \frac{\dot \tau}{\tau}\(\int R\dd
   y\)^{1/2}  \sqrt{\mathcal D(t)}. 
 \end{align*}
Similarly, at least in the case $\nu=0$, the formal
    conservation of the energy 
    $E$ defined by \eqref{eq:energy-init}, encompasses the assumption of point (c). 
\end{remark}

\begin{remark}[Wasserstein distance]
  The points (b) and (c) of Theorem~\ref{theo:temps-long} imply the
 large time convergence of $R$ to $\Gamma$ in the Wasserstein
 distance $W_2$, defined, for $\nu_1$
and $\nu_2$ probability measures, by 
\begin{equation*}
  W_p(\nu_1,\nu_2)=\inf \left\{ \left(\int_{{\mathbb R}^d\times
    {\mathbb R}^d}|x-y|^p\dd\mu(x,y)\right)^{1/p};\quad (\pi_j)_\sharp
\mu=\nu_j\right\}, 
\end{equation*}
where $\mu$ varies among all probability measures on ${\mathbb R}^d\times
{\mathbb R}^d$, and $\pi_j:{\mathbb R}^d\times {\mathbb R}^d\to
{\mathbb R}^d$ denotes the canonical 
projection onto the $j$-th factor. This implies, for instance, the
convergence of fractional momenta (see
e.g. \cite[Theorem~7.12]{Vi03}) 
\begin{equation}\label{eq:moment-int}
  \int |y|^{2s}R(t,y) \dd y\Tend t \infty \int
  |y|^{2s}\Gamma(y)\dd y,\quad 0\le s\le 1.
\end{equation} 
\end{remark}
Back to the initial unknowns $(\rho,u)$, Theorem~\ref{theo:temps-long}
and \eqref{eq:uvFluid} yield
\begin{equation*}
  \rho(t,x)\Eq t \infty
  \frac{\|\rho_0\|_{L^1(\R^d)}}{\pi^{d/2}}\frac{1}{\tau(t)^d}e^{-|x|^2/\tau(t)^2},
\end{equation*}
as announced in Theorem~\ref{thm_mainrough}, 
where the symbol $\sim$ means that only a weak limit is
considered. However, in the special case of Gaussian initial data considered in
Section~\ref{sec:explicit}, it is easy to check that all the
assumptions of Theorem~\ref{theo:temps-long} are satisfied, and moreover that
$R(t,\cdot)\to \Gamma$ strongly in $L^1(\R^d)$. Finally, another
consequence of Lemma~\ref{lem:taubis},  the (proof of the) last point in
Theorem~\ref{theo:temps-long}, and \eqref{eq:uvFluid} is 
\begin{equation*}
  \frac{1}{2}\int_{\R^d}\rho(t,x) |u(t,x)|^2 \dd x\Eq t \infty
  P'(0)d\|\rho_0\|_{L^1(\R^d)}  \, \ln t \Eq t
  \infty  -P'(0)\int_{\R^d} \rho(t,x)\ln \rho(t,x)\dd x.
\end{equation*}
This shows that indeed, no \emph{a priori}
information can be directly extracted from the energy $E$ defined in
\eqref{eq:energy-init}.

%%%%%%%%%%%%%%%%%%%%%%%%%%%%%%%%%%%%%%%%%%%%%%

\section{From Gaussians to Theorem~\ref{theo:temps-long}} 
\label{sec:longtime}

This part of the paper is devoted to the large time behavior of
solutions to \eqref{eq:euler-isotherme} and 
its variants.  We first compute explicit Gaussian solutions and then
proceed to the proof of Theorem~\ref{theo:temps-long}.

\subsection{Explicit solution}
\label{sec:explicit}

In this section, we resume and generalize some results established in
\cite{Yuen12,CFY17}. The generalizations concern two aspects: we allow
densities and velocities which are not centered at the same point
(hence $\overline x_j$ and $c_j$ below), and we consider the quantum
Navier-Stokes equation. 

\subsubsection{Euler and Newtonian Navier-Stokes equations}

We recall the compressible Euler equation for isothermal fluids on $\R^d$
\begin{equation}\label{eq:euler}
  \left\{
    \begin{aligned}
 & \d_t \rho+\DIV\(\rho u\)=0,\\
     &\d_t (\rho u) +\DIV(\rho u\otimes u)+\kappa \nabla
        \rho= 0,
    \end{aligned}
\right.
\end{equation}
where $\kappa>0$. 
%Isothermality corresponds to a pressure law which is linear with respect to the density $\rho$, that is $P(\rho)=\kappa \rho$ in the present case. 
As noticed in \cite{Yuen12}, \eqref{eq:euler} has a family of
explicit solutions with Gaussian densities and affine velocities
centered at the same point. Allowing different initial centers for
these quantities leads to considering 
\begin{equation}
  \label{eq:CIgauss}
  \rho(0,x) = b_0 e^{-\sum_{j=1}^d \alpha_{0j}x_j^2},\quad u(0,x) = 
  \begin{pmatrix}
    \beta_{01}x_1\\
\vdots\\
\beta_{0d}x_d
  \end{pmatrix}+
  \begin{pmatrix}
    c_{01}\\
\vdots\\
c_{0d}
  \end{pmatrix},
\end{equation}
with $b_0,\alpha_{0j}>0$, $\beta_{0,j},c_{0,j}\in \R$.
Seeking a
solution of the form
\begin{equation*}
  \rho(t,x) = b(t) e^{-\sum_{j=1}^d \alpha_{j}(t)\(x_j-\overline
  x_j\)^2},\quad u(t,x) = 
\begin{pmatrix}
    \beta_{1}(t)x_1\\
\vdots\\
\beta_{d}(t)x_d
  \end{pmatrix}+
  \begin{pmatrix}
    c_{1}(t)\\
\vdots\\
c_{d}(t)
  \end{pmatrix},
\end{equation*}
and plugging this ansatz into \eqref{eq:euler}, we obtain a set of
ordinary differential equations:
\begin{align}
 & \dot \alpha_j +2\alpha_j \beta_j=0, \quad
 \dot\beta_j +\beta_j^2-2\kappa \alpha_j=0, \label{eq:beta}\\
& \dot{\overline x}_j = \beta_j \overline x_j +c_j,\quad 
 \dot b = b\sum_{j=1}^d \(\dot \alpha_j {\overline x}_j^2+2
             \alpha_j {\overline x}_j \dot{\overline x}_j -2\alpha_j c_j \overline x_j
             -\beta_j \),\label{eq:b}\\
& \dot c_j +\beta_j c_j +2\kappa \alpha_j \overline x_j =
    0.\label{eq:c}
\end{align}
Mimicking \cite{LiWa06}, seeking $\alpha_j$ and $\beta_j$ of the form
\begin{equation*}
 \alpha_j(t) =\frac{\alpha_{0j}}{\tau_j(t)^2},\quad \beta_j(t) =
   \frac{\dot \tau_j(t)}{\tau_j(t)},
\end{equation*}
we check that the two equations in  \eqref{eq:beta} are satisfied if and only if
\begin{equation}
  \label{eq:tauj}
  \ddot \tau_j =
  \frac{2\kappa \alpha_{0j}}{\tau_j},\quad
    \tau_j(0)=1,\quad \dot \tau_j(0)=\beta_{0j},
\end{equation}
and we find
\begin{equation*}
  b(t)=\frac{b_0}{\prod_{j=1}^d \tau_j(t)},\quad
  \overline x_j(t)= c_{0j}t,\quad c_j(t) =c_{0j}\(1-\frac{\dot
    \tau_j(t)}{\tau_j(t)}t\). 
\end{equation*}
\begin{remark}
  Since the velocity is affine in $x$, this computation also yields
  explicit solutions for the isothermal (Newtonian) Navier-Stokes equations, but
  not for its quantum counterpart, as we will see below.
\end{remark}

\subsubsection{Korteweg and quantum Navier-Stokes equations}

 As in
\cite{CFY17}, we generalize \eqref{eq:euler} by allowing the presence
of a Korteweg term ($\eps>0$), and we extend this contribution by
allowing a quantum dissipation (quantum Navier-Stokes equation, when
$\nu>0$). We recall the isothermal Korteweg and quantum Navier-Stokes equations
\begin{equation}\label{eq:korteweg}
  \left\{
    \begin{aligned}
 & \d_t \rho+\DIV\(\rho u\)=0,\\
     &\d_t (\rho u) +\DIV(\rho u\otimes u)+\kappa \nabla
        \rho= \frac{\eps^2}{2}\rho\nabla\( \frac{\Delta
        \sqrt{\rho}}{\sqrt{\rho}}\)+\nu \DIV \(\rho D(u)\),
    \end{aligned}
\right.
\end{equation}
with $\eps,\nu\ge 0$, and where the Korteweg term is also equal to
$$\frac{\eps^2}{4}\nabla \Delta 
      \rho -\eps^2\DIV\(\nabla\sqrt\rho\otimes\nabla \sqrt\rho\),
$$
which is called the Bohm's identity.      
Proceeding as in the previous subsection,
\eqref{eq:beta}--\eqref{eq:c} become

\begin{align}
 & \dot \alpha_j +2\alpha_j \beta_j=0, \quad
 \dot\beta_j +\beta_j^2-2\kappa \alpha_j=\eps^2
        \alpha_j^2-\nu \alpha_j\beta_j, \label{eq:beta2}\\
& \dot{\overline x}_j = \beta_j \overline x_j +c_j,\quad 
 \dot b = b\sum_{j=1}^d \(\dot \alpha_j {\overline x}_j^2+2
             \alpha_j {\overline x}_j \dot{\overline x}_j -2\alpha_j c_j \overline x_j
             -\beta_j \),\notag\\
& \dot c_j +\beta_j c_j +2\kappa \alpha_j \overline x_j =
     -\eps^2\alpha_j^2 \overline x_j+\nu \alpha_j\beta_j\overline x_j.\notag
\end{align}
Again, we seek $\alpha_j$ and $\beta_j$ of the form
\begin{equation*}
 \alpha_j(t) =\frac{\alpha_{0j}}{\tau_j^{\eps,\nu}(t)^2},\quad \beta_j(t) =
   \frac{\dot \tau^\eps_j(t)}{\tau^{\eps,\nu}_j(t)},
\end{equation*}
we check that the two equations in  \eqref{eq:beta2} are satisfied if and only if
\begin{equation}
  \label{eq:tauj-quant}
     \ddot \tau^{\eps,\nu}_j =
  \frac{2\kappa
    \alpha_{0j}}{\tau^{\eps,\nu}_j}+\eps^2\frac{\alpha_{0j}^2}{(\tau^{\eps,\nu}_j)^3}-\nu
  \alpha_{0j}\frac{\dot \tau^{\eps,\nu}_j}{(\tau^{\eps,\nu}_j)^2},\quad 
    \tau^{\eps,\nu}_j(0)=1,\quad \dot \tau^{\eps,\nu}_j(0)=\beta_{0j},
\end{equation}
and we find, like before,
\begin{equation*}
  b(t)=\frac{b_0}{\prod_{j=1}^d \tau^{\eps,\nu}_j(t)},\quad
  \overline x_j(t)= c_{0j}t,\quad c_j(t) =c_{0j}\(1-\frac{\dot
    \tau^{\eps,\nu}_j(t)}{\tau^{\eps,\nu}_j(t)}t\). 
\end{equation*}

\subsubsection{A universal behavior}

It is obvious that the Euler equation \eqref{eq:euler} is a particular
case of \eqref{eq:korteweg}, by taking $\eps=\nu=0$. The Korteweg
equation ($\nu=0$)
is in turn related to the nonlinear Schr\"odinger equation, through
Madelung transform. In the present case, consider the logarithmic
Schr\"odinger equation in the semi-classical regime,
\begin{equation}
  \label{eq:logNLS}
  i\eps\d_t \psi^\eps +\frac{\eps^2}{2} \Delta \psi^\eps =\kappa
  \ln\(|\psi^\eps|^2\)\psi^\eps .
\end{equation}
The Madelung transform consists in writing the solution as
$\psi^\eps=\sqrt\rho e^{i\phi/\eps}$, with $\rho\ge 0$ and $\phi$
real-valued. Plugging this form into \eqref{eq:logNLS} and
identifying the real and imaginary parts yields
\eqref{eq:korteweg}, with the identification $u=\nabla \phi$. The
model \eqref{eq:logNLS} was introduced in 
\cite{BiMy76}, where the authors noticed that this equation possessed
explicit (complex) Gaussian solutions: the phase $\phi$ is then
quadratic, hence a velocity $u=\nabla \phi$ which is linear (or
affine). For fixed $\eps>0$, the large time dynamics for
\eqref{eq:logNLS} was studied in \cite{CaGa-p}. 
\smallbreak

As a matter of fact, the presence of a Korteweg ($\eps>0$) or quantum
Navier-Stokes ($\nu>0$) term does not alter the large time dynamics
provided in Lemma~\ref{lem:taubis}:
\begin{lemma}\label{lem:ODE2}
  Let $\alpha,\kappa>0$, $\beta\in \R$, and $\eps,\nu\ge  0$. Consider 
\begin{equation}
  \label{eq:tau-quant}
     \ddot \tau^{\eps,\nu} =
  \frac{2\kappa
  }{\tau^{\eps,\nu}}+\frac{\eps^2}{(\tau^{\eps,\nu})^3}-\nu
  \frac{\dot \tau^{\eps,\nu}}{(\tau^{\eps,\nu})^2},\quad 
    \tau^{\eps,\nu}(0)=\alpha,\quad \dot \tau^{\eps,\nu}(0)=\beta.
\end{equation}
It has a unique solution $\tau^{\eps,\nu}\in C^2(0,\infty)$, and it satisfies, as
$t\to \infty$, 
\begin{equation*}
  \tau^{\eps,\nu}(t)\Eq t \infty 2t \sqrt{\kappa \ln t} \, ,\quad \dot
  \tau^{\eps,\nu}(t)\Eq t \infty 2\sqrt{\kappa\ln  t}. 
\end{equation*} 
\end{lemma}
We present a sketchy proof of Lemma~\ref{lem:ODE2} in
Appendix~\ref{sec:ODE}.  
\smallbreak

Lemma~\ref{lem:ODE2} shows that $\eps$ and $\nu$ do not influence the
large time dynamics in \eqref{eq:tauj}. In particular, 
\begin{align*}
 &\alpha_j(t)\Eq t \infty \(2t\sqrt{\kappa \ln t}\)^{-2},\quad
 \beta_j(t)\Eq t \infty \frac{1}{t},\quad
b(t)\Eq t \infty \frac{\|\rho(0)\|_{L^1}}{\pi^{d/2}}\frac{1}{(2t\sqrt{\kappa \ln
  t})^{d}},\\
& \overline x_j(t)=c_{0j}t=o\(\alpha_j(t)\),\quad
c_j(t)\Tend t \infty 0,
\end{align*}
thus revealing some unexpected universal behavior for the explicit
solutions to \eqref{eq:korteweg}. This is an important hint to believe
in  Theorem~\ref{theo:temps-long}, as well as a precious guide in the
computation, in particular in the derivation of the change of unknown
functions \eqref{eq:uvFluid}.

\subsection{Proof of Theorem~\ref{theo:temps-long}}
\label{sec:convergence}

For the end of this section, we consider a pressure $P$ satisfying
Assumption~\ref{hyp:P-temps-long}. 
As a preamble, we prove a useful a priori estimate:

\begin{lemma}\label{lem:aprioriE}
Consider a density $R(t,y)$ and a velocity-field $U(t,y).$ Suppose that the mass $\int_{\R^d}R(t,y)\dd y$ is bounded and the pseudo-energy
\begin{equation*}
  \mathcal E(t) = \frac{1}{2\tau^2}\int_{\R^d}
  R|U|^2+\frac{\eps^2}{2\tau^2} \int_{\R^d}|\nabla \sqrt R|^2 
  +\kappa\int_{\R^d} (R|y|^2 +  R \ln R) +
  \frac{\tau^d}{\theta}\int_{\R^d} G\(\frac{ \theta R}{\tau^d}\)  
\end{equation*}
is bounded from above for positive times, $\mathcal E(t)\le
\Lambda$ for all $t\ge 0$. Then there exists $C_0>0$ such that for all
$t\ge 0$,
\begin{equation*}
\frac{1}{2\tau^2}\int_{\R^d} R|U|^2 +\frac{\eps^2}{2\tau^2} \int_{\R^d}|\nabla \sqrt R|^2
+ \kappa \int_{\R^d}  R (1+|y|^2 + |\ln  R|)  +
\tau^d\int_{\R^d} G\(\frac{ \theta R}{\tau^d}\)
\le C_0.
\end{equation*}
\end{lemma}
In view of Remark~\ref{rem:hyp-naturelles}, the assumption of this
lemma is a consequence of the assumptions of
Theorem~\ref{theo:temps-long}. 
\begin{proof}
  We note that since $P$ is convex, $G\ge0 $, so all the terms in
  $\mathcal E$ but one are non-negative. The functional
  \begin{align*}
    \mathcal E_+(t)&  := \frac{1}{2\tau^2}\int
  R|U|^2+\frac{\eps^2}{2\tau^2} \int|\nabla \sqrt R|^2 
  +P'(0)\(\int R|y|^2 + \int_{R\ge 1} R \ln R\) \\
&\quad+ \frac{\tau^d}{\theta}\int G\(\frac{\theta  R}{\tau^d}\) 
  \end{align*}
is the sum of non-negative terms, and
\begin{equation*}
  \mathcal E_+(t)\le \Lambda+ P'(0)\int_{R< 1} R \ln \frac{1}{R} .
\end{equation*}
Note that for any $\eta>0$,
\begin{equation*}
  \int_{R< 1} R \ln \frac{1}{R} \lesssim \int_{\R^d} R^{1-\eta}.
\end{equation*}
Using the interpolation inequality
\begin{equation*}
  \int_{\R^d} R^{1-\eta} \le  C_{\eta}
  \|R\|_{L^1(\R^d)}^{1-\eta-d\eta/2}\||y|^2
  R\|_{L^1(\R^d)}^{d\eta/2},\quad 0<\eta<\frac{2}{d+2},
\end{equation*}
we infer 
 \begin{equation*}
  \mathcal E_+(t)\le \Lambda+ C_\eta
  P'(0)\|R\|_{L^1(\R^d)}^{1-\eta-d\eta/2}\||y|^2
  R\|_{L^1(\R^d)}^{d\eta/2}. 
\end{equation*}
Choosing $\eta<2/(d+2)$ and invoking the boundedness of mass, we deduce
that $ \mathcal E_+(t)$ remains uniformly bounded for  $t \ge 0$. The
lemma follows by recalling the estimate used above,
\begin{equation*}
\int_{R< 1} R \ln \frac{1}{R} \lesssim
\|R\|_{L^1(\R^d)}^{1-\eta-d\eta/2}\||y|^2 R\|_{L^1(\R^d)}^{d\eta/2}. 
\end{equation*}
\end{proof}

\begin{remark}
  In view of the evolution of $\mathcal E$ given by
  \eqref{eq:evol-pseudo}, and given that the dissipation $\mathcal D$
  defined by \eqref{eq:dissip-pseudo} is non-negative, the assumptions
  of Lemma~\ref{lem:aprioriE} are fairly natural. The uniform
  boundedness of $\mathcal E$, consequence of the conclusion of
  Lemma~\ref{lem:aprioriE}, explains why the assumption
\[
\int_0^\infty \mathcal D(t)\dd t<\infty,
\]
made in Theorem~\ref{theo:temps-long}, is quite sensible, even without
invoking the  Csisz\'ar-Kullback inequality to claim that $\mathcal
E\ge 0$. 
\end{remark}

%{\color{blue}Ecrire l'equation d'evolution de la densite de
 % pseudo-energie, pour montrer quel type d'IPP on utilise pour avoir
 % $\dot {\mathcal E}+\mathcal D=0$ ? \c Ca ne joue pas un grand r\^ole
%ici\dots}

\begin{proof}[Proof of Theorem \ref{theo:temps-long}]
We assume that $(\sqrt{R},\sqrt{R}U,\mathbb T_N,\mathbb S_K)$ is a global weak solution of \eqref{eq:RU}, in the sense of Definition \ref{def:sol-globale},
with constant mass.
\smallbreak

\noindent
{\bf (a)} The proof of the first point is a rather straightforward consequence of
Lemma~\ref{lem:aprioriE}. Define
\begin{equation*}
 \mathcal I_1(t)=\int_{\R^d} (RU)(t,y)\dd y,\quad \mathcal I_2(t)= \int_{\R^d}y
  R(t,y)\dd y.
\end{equation*}
We compute
\begin{align*}
  \dot {\mathcal I}_1 & = -\frac{1}{\tau^2}\int_{\R^d}\DIV ( R U \otimes U)
      -2\kappa\mathcal  I_2 
     -\frac{\tau^d}{\theta}\int_{\R^d}\nabla P\( \tfrac{\theta R}{\tau^d}\) +
      \frac{\eps^2}{2\tau^2}\int_{\R^d} \DIV \( \mathbb S_K  \)\\
&\quad +\frac{\nu}{\tau^2} \int_{\R^d} \DIV (\sqrt R \mathbb S_N) + \frac{\nu \dot
             \tau}{\tau} \int_{\R^d}\nabla R. 
\end{align*}
By Lemma~\ref{lem:aprioriE}, all the functions whose
divergence or gradient is present above are integrable, and 
%since
%\[ 
%\int_{\R^d} R\nabla\( \frac{\Delta 
%        \sqrt{R}}{\sqrt{R}}\) = -\int_{\R^d} \nabla \(\sqrt R\)^2 \frac{\Delta 
%        \sqrt{R}}{\sqrt{R}}= -\frac{1}{2}\int_{\R^d}\nabla \left|\nabla
%      \sqrt R\right|^2,
%\]
we simply get
\[
\dot {\mathcal I}_1 = -2\kappa \mathcal I_2.
\]
Similarly, we compute
\begin{align*}
  \dot{\mathcal  I}_2 & = -\frac{1}{\tau^2}\int_{\R^d}y \DIV (RU)
             =\frac{1}{\tau^2} \int_{\R^d} RU\equiv
                        \frac{1}{\tau^2}\mathcal I_1, 
\end{align*}
where we have used the property $ R |y||U|\in L^\infty_{\rm loc}(0,\infty;L^1(\R^d))$,
which stems from  Lemma~\ref{lem:aprioriE} and Cauchy-Schwarz
inequality. Therefore,
\[ 
\dot {\mathcal I}_1 = -2\kappa \mathcal I_2,\quad \dot {\mathcal I}_2
= \frac{1}{\tau^2}\mathcal I_1. 
\]
Introducing $J_2=\tau \mathcal I_2$, we readily compute $\ddot J_2=0$, hence
\begin{equation*}
  \mathcal I_2(t) = \frac{-\mathcal I_1(0)t+\mathcal
    I_2(0)}{\tau(t)},\quad \mathcal  I_1(t) =
  \mathcal I_1(0)-2\kappa \int_0^t \mathcal I_2(s)\dd s.
\end{equation*}
The first point of Theorem~\ref{theo:temps-long} is then a direct
consequence of Lemma~\ref{lem:taubis}. 
\bigbreak

\noindent
{\bf (b)}
We split the proof of the second point into four steps.

\medskip\noindent
{\it Step 1.} We first obtain an equation satisfied by $R$ only. Since
$\partial_t(\tau^2 \partial_t R) = - \partial_t \DIV (R U)$ as well as
$\partial_t(\tau^2 \partial_t R) = \tau^2 \partial_t^2 R + 2\dot \tau
\tau \partial_t R$, we obtain from \eqref{eq:RU} that 
\[
\begin{aligned}
\tau^2 \partial_t^2 R + 2\dot \tau \tau \partial_t R
&= P'(0) \mathcal L R + \frac{1}{\tau^2} \nabla^2 : (R U \otimes U) 
+ \DIV \left( (P'(\tfrac{\theta R}{\tau^d}) - P'(0))  \cdot \nabla R   \right)\\
&\quad
%+\frac{\eps^2}{2\tau^2} \DIV \left( R \nabla \left(
%    \frac{\Delta\sqrt{R}}{\sqrt{R}} \right)   \right) 
%+\frac{\nu}{\tau^2} \nabla^2 : (R DU)
+ \nabla^2: \left( \dfrac{\nu}{\tau^2} \sqrt{R} \mathbb S_N + \dfrac{\varepsilon^2}{2\tau^2} \mathbb S_{K} \right)
+\nu \frac{\dot \tau}{\tau} \Delta R ,
\end{aligned}
\]
where we denote by $\mathcal L$ the Fokker-Planck operator $\mathcal L
R:= \Delta_y R + 2 \DIV_y (y R)$.

\medskip\noindent
{\it Step 2.} 
Since $\tau^2\ll (\tau\dot \tau)^2$ as $t\to
\infty$, it is natural to introduce the new time variable
\begin{equation*}
  s(t) = P'(0) \int\frac{1}{\tau\dot \tau}= \frac{1}{2}\int
  \frac{\ddot \tau}{\dot \tau} = \frac{1}{2}\ln \dot \tau(t)\Eq t
  \infty \frac{1}{4}\ln \ln t,
\end{equation*}
where the last estimate stems from Lemma~\ref{lem:taubis}, and we define $\alpha : s \mapsto \alpha(s) = t$. We observe that, thanks to Lemma \ref{lem:taubis}, the following asymptotics estimates hold in terms of the $s$-variable:
\[
\tau \circ \alpha(s) \underset{s\to \infty}{\sim} 2 \sqrt{P'(0)} e^{2s} e^{e^{4s}}, \quad
\dot\tau \circ \alpha(s) \underset{s \to \infty}{\sim} 2 \sqrt{P'(0)} e^{2s} .
\]

Setting $\bar R(s,y)=R(t,y),$ $\bar U(s,y)=U(t,y)$ and $\bar{\mathbb T}_N(s,y) = \mathbb T_N(t,y),$ $\bar{\mathbb S}_K(s,y) = \mathbb S_K(t,y),$  a
straightforward computation shows that $\bar R$ satisfies 
\[
\partial_s \bar R - \frac{2P'(0)}{(\dot \tau \circ \alpha)^2}\partial_s \bar R + \frac{P'(0)}{(\dot \tau \circ \alpha)^2}\partial_s^2 \bar R =  \mathcal L \bar R + \mathcal N_{\alpha} [\bar R, \bar U, \bar{\mathbb T}_N,\bar{\mathbb S}_K]  ,
\]
where
\begin{equation}\label{eq:Nalpha}
\begin{aligned}
% P'(0) pas de tel facteur ?!? 
\mathcal N_{\alpha} [\bar R, \bar U, \bar{\mathbb T}_N,\bar{\mathbb S}_K] 
&:= \frac{1}{ ( \tau \circ \alpha)^2} \nabla^2 : (\bar R \bar U \otimes \bar U)\\
&
+ \DIV \left( \(P'\(\tfrac{\theta\bar R}{(\tau \circ \alpha)^d}\) - P'(0)\)
  \nabla \bar R   \right) \\ 
%+\frac{\eps^2}{2(\tau \circ \alpha)^2} &\DIV \left( \bar R \nabla
% \left( \frac{\Delta\sqrt{\bar R}}{\sqrt{\bar R}} \right)   \right) 
%+\frac{\nu}{ (\tau \circ \alpha)^2} \nabla^2 : (\bar R D \bar U)
& + \nabla^2 : \left( \frac{\nu}{ (\tau \circ \alpha)^2} \sqrt{\bar{R}} \bar{\mathbb S}_N +  \dfrac{\varepsilon^2}{(\tau \circ \alpha)^2} \bar{\mathbb S}_K \right)
+\nu \frac{\dot \tau \circ \alpha}{\tau \circ \alpha} \Delta \bar R,
\end{aligned}
\end{equation}
and the same compatibility conditions between overlined quantities as in 
\eqref{eq_compnewton}--\eqref{eq_compkorteweg}.
We also remark that  
%since $\frac{d}{d\sigma} (G(\sigma) - \sigma
%G'(\sigma))= P'(0) - P'(\sigma)$, hence $G(\sigma) - \sigma G'(\sigma) =
%P'(0)\sigma -P(\sigma)$, 
we have
\[
\DIV \left( \(P'\(\tfrac{\theta\bar R}{(\tau \circ \alpha)^d}\) -
  P'(0)\)   \nabla \bar R   \right) 
= \dfrac{(\tau\circ \alpha)^d}{\theta} \Delta \( P(\sigma) - \sigma P'(0)
\)\Big|_{\sigma = \frac{\theta\bar R}{(\tau \circ \alpha)^d} }. 
\]
In view of Lemma~\ref{lem:aprioriE},
$(\bar R, \bar U)$ verifies for
some constant $C_0 >0$
\begin{equation}\label{eq:tilde-R-Linftyt}
\sup_{s \ge 0} \int_{\R^d} \bar R (1+|y|^2 + |\ln \bar R|) \, \dd y \le C_0,
\end{equation} 
and we also have, by the assumption $\int_0^\infty \mathcal
D(t)\dd t<\infty$, 
\begin{equation}\label{eq:tilde-R-U-L2t}
\begin{aligned}
\int_0^\infty &\left(\frac{\dot \tau \circ \alpha}{\tau \circ \alpha}
\right)^2 \left( \left\|\sqrt{\bar R} \bar U  \right\|_{L^2(\R^d)}^2  
+ \eps^2  \left\| \nabla \sqrt{\bar R}  \right\|_{L^2(\R^d)}^2 \right) \dd s \\
&+\int_0^\infty (\dot \tau\circ \alpha)^2 (\tau\circ \alpha)^{d} 
\left( \int_{\R^d} \( P(\sigma) - \sigma P'(0) \)\Big|_{\sigma = \frac{\bar
      R}{(\tau \circ \alpha)^d} } \, \dd y \right)  \dd s  \\ 
&+ \nu \int_0^\infty \frac{\dot \tau \circ \alpha}{(\tau \circ
  \alpha)^3} \left\| \bar{\mathbb S}_N\right\|_{L^2(\R^d)}^2 \, \dd s  
\le C_1,
\end{aligned}
\end{equation}
for some constant $C_1>0$.

\medskip\noindent
{\it Step 3.} 
Let $s \in [0,1]$ and consider a sequence $s_n \to \infty$ when $n \to
\infty$. Define the sequences $\bar R_n(s,y) := \bar R(s+s_n,y)$,
$\bar U_n(s,y) := \bar U (s+s_n,y),$
$\bar{\mathbb T}_{N,n} := \bar{\mathbb T}_{N}(s+s_n,y),$
$\bar{\mathbb S}_{K,n} := \bar{\mathbb S}_{K}(s+s_n,y)$ and
$\alpha_n(s) := \alpha (s + 
s_n)$, in such a way that 
\begin{equation}\label{eq:tilde-Rn-Un}
\partial_s \bar R_n - \frac{2P'(0)}{(\dot \tau\circ
  \alpha_n)^2}\partial_s \bar R_n + \frac{P'(0)}{(\dot \tau\circ
  \alpha_n)^2}\partial_s^2 \bar R_n =  \mathcal L \bar R_n + \mathcal
N_{\alpha_n} [\bar R_n ,\bar U_n, \bar{\mathbb T}_{N,n} , \bar{\mathbb S}_{K,n}]. 
\end{equation}
Moreover, estimates \eqref{eq:tilde-R-Linftyt} and \eqref{eq:tilde-R-U-L2t} yield
\begin{equation} \label{eq:tilde-Rn-Linftyt}
\sup_{n \in \N}  \sup_{s \in [0,1]} \int_{\R^d} \bar R_n (1+|y|^2 + |\ln  \bar R_n|) \dd y 
 \le C,
\end{equation} 
and
\begin{equation}\label{eq:tilde-Rn-Un-L2t}
  \begin{aligned}
 & \lim_{n \to \infty} 
    \int_0^1 \left(\frac{\dot \tau\circ \alpha_n}{\tau\circ \alpha_n}
    \right)^2 \left( \left\| \sqrt{\bar R_n} \bar U_n  \right\|_{L^2(\R^d)}^2   
+ \eps^2  \left\|  \nabla \sqrt{\bar R_n}  \right\|_{L^2(\R^d)}^2 \right) \dd s =0, \\
&\lim_{n \to \infty} \int_0^1 (\dot \tau\circ \alpha_n)^2 (\tau\circ \alpha_n)^{d} 
\left( \int_{\R^d} \( P(\sigma) - \sigma P'(0) \)\Big|_{\sigma = \frac{\theta
      \bar
      R}{(\tau \circ \alpha_n)^d} } \, \dd y \right)  \dd s =0, \\ 
& \lim_{n \to \infty} \nu \int_0^1 \frac{\dot \tau \circ \alpha_n}{(\tau \circ
  \alpha_n)^3} \left\| \bar{\mathbb S}_{N,n} \right\|_{L^2(\R^d)}^2 \, \dd
s  = 0.
  \end{aligned}
\end{equation}

From \eqref{eq:tilde-Rn-Linftyt} and Dunford-Pettis theorem, we deduce
that there exists  
\[
R_\infty \in L^1((0,1) \times \R^d)
\]
such that (up to extracting a subsequence)
\[
\bar R_n \rightharpoonup R_\infty \text{ weakly in } L^1((0,1) \times \R^d) \text{ as } n \to \infty,
\]
%as well as $\bar R_n(0)\rightharpoonup R_{0,\infty}$ weakly in $L^1(\R^d)$.
with $R_{\infty}$ of finite (mean) relative entropy $\int_0^1
\int_{\R^d} |{R}_{\infty}\ln ({R}_{\infty}/\Gamma)| < \infty.$  

Therefore, passing to the limit $n \to \infty$ in Equation \eqref{eq:tilde-Rn-Un}, we obtain
\begin{equation}\label{eq:Rinfty}
%\left\{
%\begin{aligned} &
\partial_s R_\infty = \mathcal L R_\infty  \text{ in }  \mathcal D'
((0,1) \times \R^d) .  
%\\
%&  R_\infty \big|_{s=0} = R_{0,\infty}.
%\end{aligned}
%\right.
\end{equation}
In order to establish \eqref{eq:Rinfty}, the convergence of the second, third and fourth terms of
\eqref{eq:tilde-Rn-Un} are evident, hence we only give the details for
the convergence of the term $\mathcal N_{\alpha_n} [\bar R_n ,\bar U_n, \bar{\mathbb T}_{N,n}, \bar{\mathbb S}_{K,n}] $  in $\mathcal D' ((0,1) \times \R^d)$. 
For any $\phi \in \mathcal D ((0,1) \times \R^d)$ we have
\[
\begin{aligned}
&
\left| \left\langle \frac{1}{(\tau\circ\alpha_n)^2} \nabla^2 : (\bar R_n \bar U_n \otimes \bar U_n) , \phi \right\rangle \right| 
= \left|
\int_0^1 \!\! \int_{\R^d}
\frac{1}{(\tau\circ\alpha_n)^2}   (\bar R_n \bar U_n \otimes \bar U_n) : \nabla^2 \phi \, \dd y \, \dd s \right| \\
&
\lesssim  \left(\sup_{n \in \N} \sup_{s \in [0,1]} \| \bar R_n
  \|_{L^1(\R^d)} \right)  \left( \sup_{s \in [0,1]} \frac{1}{(\dot \tau
    \circ\alpha_n)^2} \right) 
 \left( \int_0^1 \left(\frac{\dot
      \tau\circ\alpha_n}{\tau\circ\alpha_n} \right)^2 \|  \sqrt{\bar
    R_n} \bar U_n  \|_{L^2(\R^d)}^2 \, \dd s \right)  ,
\end{aligned}
\]
from which we deduce, using \eqref{eq:tilde-Rn-Linftyt} and \eqref{eq:tilde-Rn-Un-L2t}, the convergence of the first term of $\mathcal N_{\alpha_n}[\bar R_n ,\bar U_n, \bar{\mathbb T}_{N,n}, \bar{\mathbb S}_{K,n}] $, that is
$$
\lim_{n \to \infty} \frac{1}{(\tau\circ\alpha_n)^2} \nabla^2 : (\bar R_n \bar U_n \otimes \bar U_n) = 0 
 \text{ in }  \mathcal D'((0,1) \times \R^d) .
$$
For the second term, we write
\[
\begin{aligned}
&
\left| \left\langle
\DIV \left( (P'(\tfrac{\theta\bar R_n}{(\tau \circ \alpha_n)^d}) - P'(0))
  \nabla \bar R_n   \right), \phi \right\rangle \right| \\ 
&\qquad 
= \left|
\int_0^1 \!\! \int_{\R^d}
 (\tau\circ \alpha_n)^d \( P(\sigma) - \sigma P'(0) \)\Big|_{\sigma =
   \frac{\theta\bar R_n}{(\tau \circ \alpha_n)^d} } \Delta \phi \, \dd y \,
 \dd s \right| \\ 
&\qquad
\lesssim  \left(\sup_{s \in [0,1]} \frac{1}{(\dot \tau \circ
    \alpha_n)^2} \right)\left( \int_0^1 \!\! \int_{\R^d} 
 (\dot \tau \circ \alpha_n)^2 (\tau\circ \alpha_n)^d \( P(\sigma) -
 \sigma P'(0) \)\Big|_{\sigma = \frac{\theta \bar R_n}{(\tau \circ
     \alpha_n)^d} } \, \dd y \, \dd s  \right),
\end{aligned}
\]
wich again converges to $0$ thanks to \eqref{eq:tilde-Rn-Linftyt} and \eqref{eq:tilde-Rn-Un-L2t}. Concerning the third term, we recall first the compatibility condition $\overline{\mathbb S}_{K,n}
=  \sqrt{\bar{R}_n}\nabla^2 \bar R_n -  \nabla \sqrt{\bar R_n} \otimes \nabla \sqrt{\bar R_n}$ from \eqref{eq_compkorteweg}, from which we obtain
\[
\begin{aligned}
& \eps^2 \left|
\left\langle
\frac{1}{(\tau\circ\alpha_n)^2} \nabla^2 : \bar{\mathbb{S}}_{K,n} ,  \phi \right\rangle \right| \\
&\qquad
\lesssim \eps^2 
\int_0^1 \!\! \int_{\R^d}
\frac{1}{(\tau\circ\alpha_n)^2}  \left( |\nabla \sqrt{\bar R_n}|^2 + |\bar R_n|   \right) (|\nabla^2 \phi | + |\nabla^3 \phi|) \, \dd y \, \dd s  \\
&\qquad
\lesssim \eps^2  \left(\sup_{s \in [0,1]} \frac{1}{(\dot \tau \circ\alpha_n)^2} \right)  \left( \int_0^1 \left(\frac{\dot \tau \circ\alpha_n}{\tau \circ\alpha_n} \right)^2 \|  \nabla \sqrt{\bar R_n} \|_{L^2(\R^d)}^2 \, \dd s \right) \\
&\qquad \quad
+ \eps^2  \left(\sup_{s \in [0,1]} \frac{1}{ (\tau\circ\alpha_n)^2} \right)  \sup_{n \in \N} \sup_{s \in [0,1]} \|  {\bar R_n}  \|_{L^1(\R^d)}  ,  
\end{aligned}
\]
which goes to $0$ by using \eqref{eq:tilde-Rn-Linftyt} and \eqref{eq:tilde-Rn-Un-L2t}. 
For the fourth term of $\mathcal N_{\alpha_n} [\bar R_n ,\bar U_n , \bar{\mathbb T}_{N,n}, \bar{\mathbb S}_{K,n}]$, we have
\[
\begin{aligned}
& \left| \nu \left\langle
 \frac{1}{(\tau\circ\alpha_n)^2} \nabla^2 : (\sqrt{\bar{R}_n} \bar{\mathbb S}_{N,n}) , \phi
\right\rangle \right| 
= \nu \left|
\int_0^1\int_{\R^d}
\frac{1}{(\tau\circ\alpha_n)^2}  \sqrt{\bar{R}_n} \bar{\mathbb S}_{N,n}: \nabla^2 \phi \, \dd y \, \dd s \right| \\
&\quad
\lesssim \nu \left|
\int_0^1\int_{\R^d}
\frac{1}{(\tau\circ\alpha_n)^2}  \sqrt{\bar R_n} \left|\bar{\mathbb S}_{N,n}\right| |\nabla^2 \phi|  \, \dd y \, \dd s \right| \\
&\quad
\lesssim \nu  \left( \int_0^1 \frac{1}{(\tau \circ\alpha_n) (\dot \tau
    \circ\alpha_n)} \left\|\sqrt{\bar R_n} \right\|_{L^2(\R^d)}^2 \, \dd s
\right)^{\frac12}   \left( \int_0^1 \frac{\dot \tau \circ\alpha_n}{(\tau
    \circ\alpha_n)^3}  \left\| \bar{\mathbb{S}}_{N,n}
  \right\|_{L^2(\R^d)}^2 \, \dd s \right)^{\frac12} \\ 
&\quad
\lesssim \nu  \left( \sup_{n\in\mathbb N}\sup_{s \in [0,1]} \frac{1}{(\dot\tau\circ\alpha_n)\,(\tau\circ\alpha_n)} \right) \left( \sup_{n \in \N}\sup_{s \in [0,1]} \| \bar R_n
  \|_{L^1(\R^d)}  \right) 
   \left( \int_0^1 \frac{\dot \tau
    \circ\alpha_n}{(\tau \circ\alpha_n)^3}  \| \bar{\mathbb{S}}_{N,n}  \|_{L^2(\R^d)}^2 \, \dd s \right)^{\frac12}  ,
\end{aligned}
\]
and that last expression converges to $0$. Finally, for the last term of $\mathcal N_{\alpha_n} [\bar R_n ,\bar U_n , \bar{\mathbb T}_{N,n}, \bar{\mathbb S}_{K,n}]$, we obtain
\[
\begin{aligned}
 \nu \left| \left\langle 
\frac{\dot\tau\circ\alpha_n}{\tau\circ\alpha_n} \Delta \bar R_n , \phi
\right\rangle \right| 
&= \nu \left| 
\int_0^1\int_{\R^d}
\frac{\dot\tau\circ\alpha_n}{\tau\circ\alpha_n} \bar R_n \Delta \phi \, \dd y \, \dd s \right| \\
& \lesssim \nu  \left( \sup_{s \in [0,1]} \frac{\dot\tau\circ\alpha_n}{\tau\circ\alpha_n} \right) \left( \sup_{n \in \N} \sup_{s \in [0,1]} \| \bar R_n \|_{L^1(\R^d)} \right)  ,
\end{aligned}
\]
which also goes to $0$.

\medskip\noindent
{\it Step 4.} 
We now follow the arguments of \cite{CaGa-p} in order to show that
$R_\infty = \Gamma$, which concludes the proof of point (b). Because $R_{\infty}$ has finite
entropy and, by a tightness argument, $\bar R_n$ cannot  
lose mass thanks to \eqref{eq:tilde-Rn-Linftyt}, \cite[Corollary
2.17]{AMTU01} entail that the solution to \eqref{eq:Rinfty} satisfies
\[
\|R_\infty(s)-\Gamma\|_{L^1(\R^d)}\Tend s \infty 0, 
\]
since $R_{\infty}$ and $\Gamma$ have the same mass in view of
\eqref{eq:uvFluid}. 
 On the other hand, in the $s$-variable we have 
$$
\partial_s \bar R + \frac{\dot \tau \circ \alpha}{P'(0) \tau \circ \alpha} \DIV (\bar R \bar U) = 0,
$$
and \eqref{eq:tilde-Rn-Un-L2t} implies
$$
\frac{\dot \tau \circ \alpha}{\tau \circ \alpha} \DIV (\bar R_n \bar U_n) \to 0 \text{ in } 
L^2((0,1) ; W^{-1,1}(\R^d)) \text{ as } n \to \infty.
$$
Therefore $\partial_s R_\infty = 0$, hence 
$R_\infty=\Gamma$. Since the limit is unique, no extraction of a
subsequence is needed, and the result does not depend on the sequence
$s_n\to \infty$, hence the result. 
\bigbreak

\noindent
{\bf (c)} The last point of Theorem~\ref{theo:temps-long} is proven by
rewriting the energy $E$, defined by \eqref{eq:energy-init}, in terms
of the new unknowns $(R,U)$ {\em via} \eqref{eq:uvFluid}:
\begin{align*}
  E(t) &= \frac{1}{2}\int \rho(t,x)|u(t,x)|^2 \dd x+\frac{\eps^2}{2} \int |\nabla
    \sqrt{\rho(t,x)}|^2 \dd x+\int F(\rho(t,x)) \dd x\\
&= \frac{\theta}{2\tau^2}\int R(t,y)|U(t,y)|^2\dd y +\frac{\theta\dot
  \tau}{2}\int R(t,y)|y|^2\dd y +\theta \frac{\dot \tau}{\tau}\int
  R(t,y)y\cdot U(t,y)\dd y\\
&\quad +\theta\frac{\eps^2}{2}\int \left|\nabla
  \sqrt{R(t,y)}\right|^2\dd y +\int
  F\(\frac{\theta}{\tau^d}R\(t,\frac{x}{\tau^d}\)\)\dd x.
\end{align*}
Recalling the identity \eqref{eq:FG},
\[ 
F(\rho) = P'(0)\rho \ln \rho + G(\rho),
\]
we can write
\begin{align*}
  E(t) &= \frac{\theta}{2\tau^2}\int R(t,y)|U(t,y)|^2\dd y +\frac{\theta\dot
  \tau}{2}\int R(t,y)|y|^2\dd y +\theta \frac{\dot \tau}{\tau}\int
  R(t,y)y\cdot U(t,y)\dd y\\
&\quad +\theta\frac{\eps^2}{2}\int \left|\nabla
  \sqrt{R(t,y)}\right|^2\dd y +\theta P'(0) \int R(t,y)\ln R(t,y)\dd
  y\\
&\quad 
  +\theta P'(0)\ln\frac{\theta}{\tau^d}\int R(t,y)\dd y 
+ \tau^d\int G\(\frac{\theta}{\tau^d}R(t,y)\)\dd y.
\end{align*}
In view of Lemma~\ref{lem:aprioriE}, the first, fourth, fifth and last
terms are bounded functions of time. Invoking in addition
Cauchy-Schwarz inequality, the third term is $\O(\dot
\tau)=\O(\sqrt{\ln t})$ from Lemma~\ref{lem:taubis}. Therefore, since
we have assumed $E(t)=o(\ln t)$, we infer
\[
\frac{\dot
  \tau}{2}\int R(t,y)|y|^2\dd y - P'(0)\ln\tau^d\int R(t,y)\dd y
=o(\ln t) \quad \text{as }t\to \infty. 
\]
Lemma~\ref{lem:taubis} yields
\[
\dot \tau = 2\sqrt{P'(0)\ln t}\(1+o(1)\),\quad \ln \tau =\(1+o(1)\)\ln t,
\]
therefore
\[
2\int R(t,y)|y|^2\dd y-d\int R(t,y)\dd y \Tend t \infty 0. 
\]
Recalling that the mass is conserved, and an easy property of the
Gaussian $\Gamma$,
\[
\int R(t,y)\dd y = \int R(0,y)\dd y = \int \Gamma(y)\dd y,\quad \int
|y|^2\Gamma(y)\dd y = \frac{d}{2} \int \Gamma(y)\dd y,
\]
the proof of the last point of Theorem~\ref{theo:temps-long}
follows. 
\end{proof}

%%%%%%%%%%%%%%%%%%%%%%%%%%%%%%%%%%%%%%%%%%%%%%

\section{On the notion of weak solutions}
\label{sec:compactness}
In this part, we investigate the notion of weak solutions that we consider for the long-time analysis.
At first, we provide {\em a priori} estimates satisfied by classical
solutions to \eqref{eq:NSK} such that 
the density decays sufficiently fast at infinity. 
These estimates justify the regularity statements of
Definition~\ref{def_weaksol}. Second, we prove sequential compactness
of bounded sets of weak solutions.  
Classically, this compactness result is a cornerstone for obtaining
existence of weak solutions.
% Nevertheless, the approximation scheme
%remains to be constructed. This is left for future research. 

\subsection{A priori estimates}
\label{sec:apriori}

In this section, we present some a priori estimates that motivate our definition of weak solution.
As we noticed before, the structure of \eqref{eq:RU} suits better {\em
  a priori} estimates than \eqref{eq:NSK}. So, from now
on, we consider solutions $(R,U)$ of the system written in this form.

\subsubsection{Energy estimate} 
First, we have an extension of Lemma~\ref{lem:aprioriE}. 

\begin{proposition}\label{prop:apriori-energy}
Consider $\eps, \nu \ge 0$. Assume that the initial data satisfies
$$
 R_0(1+|y|^2 + \ln R_0) \in L^1, \quad 
 \sqrt{R_0} U_0  \in L^2  , \quad
\eps \nabla \sqrt{R_0} \in L^2.
$$
Let $(R,U)$ be a smooth solution to \eqref{eq:RU} associated to the initial data $(R_0,U_0)$, then 
\begin{equation}\label{eq:apriori-energy}
\begin{aligned}
 R(1+|y|^2 + \ln R) &\in L^\infty( \R^+ ; L^1(\R^d)), \\
 \frac{1}{\tau} \sqrt{R} U &\in  L^\infty( \R^+ ; L^2 (\R^d)), \\ 
 \frac{\eps}{\tau} \nabla \sqrt{R} &\in L^\infty( \R^+ ; L^2 (\R^d)), \\
 \sqrt{\frac{\dot \tau}{\tau^3}} \sqrt R U &\in L^2( \R^+ ; L^2 (\R^d)), \\
 \eps \sqrt{\frac{\dot \tau}{\tau^3}} \nabla \sqrt{R} &\in L^2( \R^+ ; L^2 (\R^d)), \\
 \frac{\sqrt \nu}{\tau^2}  \sqrt{R} DU &\in L^2( \R^+ ; L^2 (\R^d)).
\end{aligned}
\end{equation}

\end{proposition}

\begin{proof}
In the same vein as in Remark~\ref{rem:hyp-naturelles}, we first remark that
$$
\frac{\nu \dot \tau}{\tau^3} \int R |\DIV U| 
\le C \frac{\nu \dot \tau^2}{\tau^2} \int R 
+ \frac12 \frac{\nu}{\tau^4} \int R |DU|^2
\le
C \frac{\nu \dot \tau^2}{\tau^2} \int R 
+ \frac12 \mathcal D.
$$
Therefore, using \eqref{eq:evol-pseudo}, the conservation of mass and
the fact that $\int \frac{\dot \tau^2}{\tau^2} < \infty$, we obtain
that the pseudo energy $\mathcal E$ is uniformly bounded from
above. Lemma~\ref{lem:aprioriE}  implies that  $\mathcal E$ is
uniformly bounded in time, and
its nonnegative dissipation $\mathcal D$ is integrable in time, which
gives the desired a priori bounds on $(R,U)$. 
\end{proof}

\subsubsection{Pseudo entropy and effective velocity}
Contrary to the Newtonian case, the previous a priori estimate is not
sufficient to run  
a classical compactness argument for proving existence of
solutions. So, we provide 
here a further estimate satisfied by a ``pseudo entropy'' of an
effective velocity. This construction 
is inspired of \cite{AS-comp}. 
Given $\lambda \in \R,$ 
% $(R,U)$ be a smooth solution to \eqref{eq:RU} and 
define the effective velocity 
$$
W_\lambda = U + \lambda \nabla \ln R.
$$
Then the pair $(R,W_\lambda)$ satisfies
\begin{equation}\label{eq:RWlambda}
\left\{
\begin{aligned}
& \partial_t R + \frac{1}{\tau^2} \DIV (R W_\lambda) = \frac{\lambda}{\tau^2} \Delta R , \\
& \partial_t (R W_\lambda) + \frac{1}{\tau^2} \DIV (R W_\lambda \otimes W_\lambda) 
+ 2 P'(0) y R + P'\(\tfrac{\theta R}{\tau^d}\)  \nabla R \\
&\qquad
= \frac{\lambda_1}{2\tau^2} R \nabla \left( \frac{\Delta \sqrt R}{\sqrt R} \right)
+\frac{\lambda_2}{\tau^2} \DIV(R D W_\lambda) 
+\frac{\lambda}{\tau^2} \Delta(R W_\lambda) 
+ \frac{\nu \dot \tau}{\tau}\nabla R ,
\end{aligned}
\right.
\end{equation}
where $\lambda_1 := 4\lambda^2 - 4\nu \lambda + \eps^2$ and $\lambda_2 := \nu - 2 \lambda$.

\begin{remark} \label{rem:RWlambda-eps}
If $0 \le \eps < \nu$, we define $\lambda := \frac{\nu - \sqrt{\nu^2 - \eps^2}}{2} \ge 0$ so that $\lambda_1 = 0$ and $\lambda_2 = \sqrt{\nu^2 - \eps^2} >0$. Then $(R, W_{\lambda})$ satisfies
\begin{equation}\label{eq:RWlambda-eps}
\left\{
\begin{aligned}
& \partial_t R + \frac{1}{\tau^2} \DIV (R W_{\lambda}) = \frac{\lambda}{\tau^2} \Delta R , \\
& \partial_t (R W_{\lambda}) + \frac{1}{\tau^2} \DIV (R W_{\lambda} \otimes W_{\lambda}) 
+ 2 P'(0) y R + P' \left( \tfrac{\theta R}{\tau^d} \right) \nabla R \\
&\qquad
= \frac{\lambda_2}{\tau^2} \DIV(R D W_{\lambda}) 
+\frac{\lambda}{\tau^2} \Delta(R W_{\lambda}) 
+ \frac{\nu \dot \tau}{\tau}\nabla R .
\end{aligned}
\right.
\end{equation}
We observe that when $\eps = 0$, then $W_{\lambda} = U$, and
\eqref{eq:RWlambda-eps} is just the original equation \eqref{eq:RU}
with $\eps = 0$. 
\end{remark}

%\begin{remark}
%If $0 < \nu <  \eps$, we define $\lambda := \nu/2$ so that $ \lambda_1 = \eps^2 -\nu^2 >0$. Then $(R, W_{\lambda})$ satisfies
%\begin{equation}\label{eq:RWlambda}
%\left\{
%\begin{aligned}
%& \partial_t R + \frac{1}{\tau^2} \DIV (R W_{\lambda}) = \frac{\lambda}{\tau^2} \Delta R , \\
%& \partial_t (R W_\lambda) + \frac{1}{\tau^2} \DIV (R W_\lambda \otimes W_\lambda) 
%+ 2 P'(0) y R + P' \( \tfrac{R}{\tau^d}   \) \nabla R \\
%&\qquad
%= \frac{\lambda_1}{2\tau^2} R \nabla \left( \frac{\Delta \sqrt R}{\sqrt R} \right)
%+\frac{\lambda}{\tau^2} \Delta(R W_\lambda) 
%+ \frac{\nu \dot \tau}{\tau}\nabla R .
%\end{aligned}
%\right.
%\end{equation}
%\end{remark} 

We define the pseudo $\lambda$-entropy of $(R,U)$ by
\begin{equation*}
\mathcal E_\lambda := 
\frac{1}{2\tau^2}\int R|W_\lambda|^2
+\frac{\lambda_1}{2\tau^2} \int|\nabla \sqrt R|^2 
+ P'(0)\int (R|y|^2 +  R \ln R) 
+ \frac{\tau^d}{\theta}\int G\(\frac{\theta R}{\tau^d}\) ,
\end{equation*}
and its associated  dissipation
\begin{equation*}
\begin{aligned}
\mathcal D_\lambda  
& := \frac{\dot \tau}{\tau^3}\int  \left\{ R |W_\lambda|^2 + \lambda_1 |\nabla \sqrt R|^2 \right\}
+ d\frac{\dot \tau}{\tau} \tau^d \(\int \left[ P\(\sigma\) -\sigma P'(0) \right] |_{\sigma= \frac{\theta R}{\tau^d}} \) \\
&\quad 
+  \frac{\lambda_2}{\tau^4} \int R |D W_\lambda|^2
+  \frac{\lambda}{\tau^4} \int R |\nabla W_\lambda|^2
+  \frac{4\lambda P'(0)}{\tau^2} \int  |\nabla \sqrt R|^2 \\
&\quad 
+  \frac{4\lambda}{\tau^2} \int \frac{\theta R}{\tau^d} G'' \( \tfrac{\theta R}{\tau^d}   \) |\nabla \sqrt R|^2
+ \frac{\lambda \lambda_1}{4 \tau^4} \int R |\nabla^2 \ln R |^2.
\end{aligned} 
\end{equation*}
 This is nothing but the pseudo energy and dissipation associated to 
 \eqref{eq:RWlambda}.
%Observe that the BD-entropy corresponds to the $\lambda$-entropy with $\lambda=\nu$.  

By reproducing computations of the a priori estimate to this system, we obtain:

\begin{lemma}\label{lem:pseudo-lambda-entropy}
Let $(R,U)$ be a smooth solution to \eqref{eq:RU} and consider the effective velocity $W_\lambda := U + \lambda \nabla \ln R$ with $\lambda \in \mathbb R$. Then the pseudo $\lambda$-entropy $\mathcal E_{\lambda}$ satisfies
$$
\frac{\dd}{\dd t} \mathcal E_{\lambda} + \mathcal D_{\lambda} 
= \frac{2  \lambda d P'(0)}{\tau^2} \int R
  - \nu \frac{\dot \tau}{\tau^3} \int R \DIV W_\lambda.
$$
\end{lemma}

\begin{remark}
If we set $\lambda = \nu,$ we note that $\lambda_1 = \varepsilon^2$ and $\lambda_2=-\nu.$ In this case, two terms in the dissipation combine to yield:
\[
\begin{aligned}
 \frac{\lambda_2}{\tau^4} \int R |D W_\lambda|^2 +
 \frac{\lambda}{\tau^4} \int R |\nabla  W_\lambda|^2
& =  \dfrac{\nu}{\tau^4} \int_{R} |A W_{\lambda}|^2 
  = \dfrac{\nu}{\tau^4} \int_{R} |A U|^2.
\end{aligned}
\]
Thus, we recover the BD-entropy estimate associated with our system
(see \cite{BD04,Bre-De-CKL-03} for the introduction of this method, and
\cite{VasseurYu} in the case with a Bohm potential).  
We apply only 
this particular case to obtain the regularity statement
below. Nevertheless, we apply  the choice of $\lambda$ from
Remark~\ref{rem:RWlambda-eps} in the next subsection (to prove a Mellet-Vasseur
estimate). This choice enables to delete the Korteweg term, this is
why we
provided a statement with a general $\lambda$ in the previous lemma. 
 
\end{remark}

Combining the latter entropy estimate with energy estimate yields
controls on $(R,U)$, 
which enable to consider various cases for the parameters
$\varepsilon$ and $\nu.$ 
This ensures the following regularity properties of a classical solution:
\begin{proposition}\label{prop:apriori-BDentropy}
Consider $\eps, \nu \ge 0$. Assume that the initial data satisfies 
$$
 R_0(1+|y|^2 + \ln R_0) \in L^1(\R^d), \quad
 \sqrt{R_0} U_0 \in L^2(\R^d) , \quad
(\eps+\nu) \nabla \sqrt{R_0} \in L^2(\R^d). 
$$
Let $(R,U)$ be a smooth solution to \eqref{eq:RU} associated to the initial data $(R_0,U_0)$, then
\begin{equation}\label{eq:apriori-BDentropy}
\begin{aligned}
 R(1+|y|^2 + \ln R) &\in L^\infty( \R^+ ; L^1(\R^d)), \\
 \frac{1}{\tau} \sqrt{R} U &\in  L^\infty( \R^+ ; L^2 (\R^d)), \\ 
 \frac{(\eps+\nu)}{\tau} \nabla \sqrt{R} &\in L^\infty( \R^+ ; L^2 (\R^d)), \\
 \sqrt{\frac{\dot \tau}{\tau^3}} \sqrt R U &\in L^2(\R^+ ; L^2 (\R^d)), \\
 \eps \sqrt{\frac{\dot \tau}{\tau^3}} \nabla \sqrt{R} &\in L^2( \R^+ ; L^2 (\R^d)), \\
 \frac{\sqrt \nu}{\tau} \nabla \sqrt{R} &\in L^2( \R^+ ; L^2 (\R^d)), \\
 \frac{\sqrt \nu}{\tau^2}  \sqrt{R} \nabla U &\in L^2( \R^+ ; L^2 (\R^d)), \\
 \frac{ \nu}{\tau^2}  \frac{\theta R}{\tau^d} G'' \(  \frac{\theta
  R}{\tau^d} \) |\nabla \sqrt R|^2 &\in L^1(\R^+;L^1(\R^d)),\\
\frac{\eps \sqrt \nu }{\tau^2} \sqrt R \nabla^2 \log R &\in L^2( \R^+
 ; L^2 (\R^d)),
\end{aligned}
\end{equation}
and we observe that last estimate implies (see \cite{Jungel,VasseurYu})
$$
\frac{\eps \sqrt \nu }{\tau^2} \,  \nabla^2 \sqrt R  \in L^2( \R^+ ; L^2 (\R^d)), 
\quad 
\frac{\eps^{1/2} \nu^{1/4} }{\tau} \, \nabla R^{\frac14} \in L^4( \R^+ ; L^4 (\R^d)).
$$
\end{proposition}

\begin{proof}[Proof of Proposition~\ref{prop:apriori-BDentropy}]
From \eqref{eq:evol-pseudo} and Lemma~\ref{lem:pseudo-lambda-entropy} with $\lambda = \nu,$ it follows
$$
\frac{\dd}{\dd t} \( \mathcal E(t) + \mathcal E_{\nu}(t)  \)
+ \mathcal D(t) + \mathcal D_{\nu}(t) = \frac{2 d \nu}{\tau^2} \int R,
$$
and we conclude in a similar way as in the proof 
of Proposition~\ref{prop:apriori-energy}, using now that $\int
dt/\tau(t)^{2} < \infty$ and recalling that $P'( \sigma ) = P'(0) + \sigma G''( \sigma )$ with $P'(0) >0$ and $G'' \ge 0$. 
\end{proof}

\subsubsection{Mellet-Vasseur estimate}
\label{sec:mellet-vasseur}

It turns out that the above estimates are insufficient for the construction of solutions to \eqref{eq:NSK} via 
a compactness approach: the above information do not enable to pass to the limit in the convective term $R U \otimes U$
(see the introduction of \cite{VasseurYuInventiones} for more precise statements). 
So, when $0 \le \varepsilon < \nu$, we add a further estimate that we
adapt from \cite{MelletVasseur,AS-comp} to the isothermal case.
For this, we restrict from now to the isothermal case $P(\rho)\equiv \kappa \rho.$ 

\begin{proposition}\label{prop:inegaliteMV}
Let $\nu >0$ and $0 \le \varepsilon \le \nu$, 
$P(\rho) = \kappa \rho$ with $\kappa >0$, and $T>0$. Assume that the initial
data satisfy
\[
R_0(1+|y|^2 + \ln R_0) \in L^1 (\R^d), \quad
 \sqrt{R_0} U_0 \in L^2(\R^d) , \quad
(\eps+\nu) \nabla \sqrt{R_0} \in L^2(\R^d). 
\]
 Let $(R,U)$ be a smooth
solution to \eqref{eq:RU} associated to the initial data
$(R_0,U_0)$. Consider $\lambda(\eps) := (\nu - \sqrt{\nu^2 -
  \eps^2})/2 \ge 0$ and define the effective velocity  
$$
W_\eps := U + \lambda(\eps) \nabla \ln R,
$$
so that $(R,W_\eps)$ satisfies \eqref{eq:RWlambda-eps}. Denote $\varphi_{MV}(z) = (1+z)\ln(1+z)$ for $z \ge 0$, and suppose further that
$$
\int_{\R^d} R_0 \varphi_{MV}\left(|W_{\eps,0}|^2 + |y|^2\right) \dd y < \infty.
$$
Then there exists a constant $K_T$ depending only on $T$ and $C''_0$ depending only on initial data such that
\begin{multline*}
\sup_{t \in (0,T)} \left\{ \int_{\R^d} R \varphi_{MV}\left(|W_\eps|^2
    + |y|^2\right)  \dd y\right\} \\ 
+ \int_{0}^T \int_{\R^d} R \varphi'_{MV}\left(|W_\eps|^2+ |y|^2\right) \{ \lambda(\eps) |\nabla W_\eps|^2 + \lambda_2(\eps)|D(W_\eps)|^2 \} \, \dd y \, \dd t
\le K_T C''_0 ,
\end{multline*}
where $\lambda_2(\eps) := \sqrt{\nu^2 - \eps^2} \ge 0$.
\end{proposition}

\begin{remark}
  The functional is not quite the same as in \cite{MelletVasseur},
  where the authors analyze $\varphi_{MV}(|u|^2)$. Considering an
  effective velocity follows from \cite{AS-comp}. On
  the other hand, the introduction of the term $|y|^2$ is due to the
  presence of term $yR$ in \eqref{eq:RU}, which is a specific feature
  of our approach adapted to the isothermal case, and seems necessary
  in order to obtain closed estimates. 
\end{remark}

\begin{proof}
We first remark that, by construction, we have:
\begin{equation} \label{eq_boundphiMV}
\varphi'_{MV}(z) = 1 + \ln(1+{z}), \qquad z |\varphi''_{MV}(z)| \le 1,
\quad \forall \, z \ge 0.   
\end{equation}
Given $(R,W_\eps)$ a solution to \eqref{eq:RWlambda-eps}, we have then:
\begin{multline*}
\dfrac{{\rm d}}{{\rm d}t}\int R \, \varphi_{MV}\left(|W_\eps|^2+ |y|^2\right) 
\\
= \int \partial_t R \, \varphi_{MV}\left(|W_\eps|^2+ |y|^2\right) + 
 2\int R \varphi'_{MV}\left(|W_\eps|^2 + |y|^2\right) W_\eps \cdot \partial_t W_\eps   .
\end{multline*} 
For conciseness, we drop the arguments of $\varphi_{MV}$ and its derivative in what follows. We may then rewrite the last integral on the right-hand side by applying that
\[
\begin{aligned}
R \partial_t W_\eps 
&=\dfrac{\lambda(\eps)}{\tau^2} [\Delta (R W_\eps) - (\Delta R) W_\eps ] 
+\dfrac{\lambda_2(\eps)}{\tau^2}{\rm div} (R D(W_\eps)) \\
&\quad
+ \dfrac{\nu \dot{\tau}}{\tau} \nabla R - \kappa \nabla R - 2\kappa y R
- \dfrac{1}{\tau^2} R W_\eps \cdot \nabla W_\eps, 
\end{aligned}
\]
from which we obtain 
\begin{align} \label{eq_debase}
 \dfrac{{\rm d}}{{\rm d}t}\int R \, \varphi_{MV}
& = \dfrac{\lambda(\eps)}{\tau^2} I_0
 +2\dfrac{\lambda_2(\eps)}{\tau^2}  I_1 \\
&  \qquad + 2\left(\dfrac{\nu \dot{\tau}}{\tau} - \kappa\right) I_2 + 2\left( \dfrac{1}{\tau^2}
   - 2 \kappa \right) I_3 ,\notag
\end{align}
with 
\begin{align*}
 I_0 &= \int   \left\{ 2[\Delta (R W_\eps) - (\Delta R) W_\eps ]\cdot W_\eps  \varphi'_{MV}  + \Delta R \varphi_{MV} \right\} ,\\
 I_1 & = \int {\rm div} (R D(W_\eps)) \cdot W_\eps  \varphi'_{MV} ,\\
 I_2 & = \int  \nabla R \cdot W_\eps \varphi'_{MV}, \\
 I_3 & = \int y R \cdot W_\eps  \varphi'_{MV} .
\end{align*}
We compute bounds above for these integrals by applying standard transformations
and application of \eqref{eq_boundphiMV}. By integrating by parts we
obtain
\begin{align*}
I_0 & =  - \int R |\nabla (|W_\eps|^2)|^2 \varphi''_{MV} - 2\int R |\nabla W_\eps|^2 \varphi'_{MV} + 2\int R (\varphi'_{MV} + 2|y|^2 \varphi''_{MV}) \\
& =  - \int R |\nabla (|W_\eps|^2)|^2 \varphi''_{MV} - 2\int R |\nabla W_\eps|^2 \varphi'_{MV} 
+\O\left( \int R (1+ |y|^2 + |W_\eps|^2) \right) .
\end{align*}
For the term $I_1$ we have
\begin{align*}
I_1 & =  - \int R \varphi'_{MV}|D(W_\eps)|^2 + \O\left( \int R (|W_\eps|^2+ |y|^2) |\varphi''_{MV}| |D(W_\eps)| |\nabla W_\eps|\right)  \\
		& = - \int R \varphi'_{MV}|D(W_\eps)|^2 +  \O\left( \int R  |\nabla W_\eps|^2\right).
\end{align*}
We compute $I_{2}$ by integrating by parts, which gives
\begin{align*}
I_{2} & =  -  \int R \, {\rm div} \, W_\eps \, \varphi'_{MV} -  2\int R
      \left[ (W_\eps \cdot \nabla) W_\eps \right] \cdot W_\eps
      \varphi''_{MV}  -2 \int R W_\eps \cdot y \varphi''_{MV}   ,
\end{align*}
and introducing an absolute constant $C$ and a small parameter $\eta >0$ to be fixed later on, we obtain
\begin{align*}
|I_{2}|	& \le C\left( \int R |D(W_\eps)| \varphi'_{MV} + \int R (|W_\eps|^2 + |y|^2) \varphi''_{MV} |\nabla W_\eps|\right) \\
	& \le \eta \int R \varphi'_{MV} |D(W_\eps)|^2 + \dfrac{C}{\eta} \left( \int R |\nabla W_\eps|^2 + \int R (1+ \varphi'_{MV} )\right)   \\
		& \le \eta \int R \varphi'_{MV} |D(W_\eps)|^2 + \dfrac{C}{\eta} \left( \int R |\nabla W_\eps|^2 + \int R (1+ |y|^2 + |W_\eps|^2) \right).
\end{align*}

Concerning  $I_3,$  Young inequality yields
\begin{align*}
|I_3|  & \le \int R  |y| |W_\eps| \left(\ln\left(1+|W_\eps|^2 + |y|^2\right) + 1\right)   
		 \le \int R \left( |y|^2 + |W_\eps|^2\right) + \int R \varphi_{MV}.
\end{align*}
We substitute $I_0$, $I_1,$ $I_2$ and $I_3$ with these computations into
\eqref{eq_debase}, and we remark that $\tau$ and $1/\tau$ are uniformly
bounded with their derivatives on $[0,T].$  We obtain then that there
exist positive constants $c_{T}$ and $C_T$ depending only on $T$ for
which 
\begin{multline} \label{eq_debase2} 
 \dfrac{{\rm d}}{{\rm d}t}\int R \, \varphi_{MV}
 + c_T   \int R \varphi'_{MV} \{ \lambda(\eps) |\nabla W_\eps|^2 +  \lambda_2(\eps) |D(W_\eps)|^2 \}
 \le \(C_T(\nu +1) + \kappa\) \times \\
\times\(\eta\int R \varphi'_{MV}|D(W_\eps)|^2 + 
 \dfrac{1}{\eta} \left(\int R|\nabla W_\eps|^2 + \int R(1+|y|^2 +|W_\eps|^2) + \int R \varphi_{MV} \right)  \)  .
\end{multline}
Choosing $\eta$ sufficiently small so the first term of the right hand
side is absorbed by the left hand side, we are in position to apply a
Gronwall lemma 
to $\int R \varphi_{MV}.$ We note here that combining the estimates of Propositions~\ref{prop:apriori-energy} and \ref{prop:apriori-BDentropy} entails 
\[
\int_0^T  \!\int_{\R^d} R |\nabla W_\eps|^2 + \int_0^T \! \int_{\R^d}
R(1 +|y|^2 + |W_\eps|^2)  \le K_T (C_0 + C'_0) , 
\]
and we obtain
\[
\sup_{t \in (0,T)} \int_{\R^d} R \varphi_{MV} \le K_T C''_0.
\]
It remains to integrate \eqref{eq_debase2} to conclude.
\end{proof} 

\begin{remark}
We note that, when $\varepsilon >0,$ the choice $\varphi_{MV}(z) =
(1+z)\ln(1+z)$ is not unique.  Indeed, with the term $I_0$ we control
a full gradient of $W_{\varepsilon}$, while the term $I_1$ only enables
a control of the symmetric part of this gradient. Hence, when
$\varepsilon =0$ we have to choose an entropy such that $z
\varphi^{''}_{MV}$ is bounded, and the parasite term appearing in $I_1$
is controlled with the previous pseudo-entropy estimate. On the other
hand, when 
$\varepsilon >0$, this Mellet-Vasseur estimate is self-consistent 
and  we can afford entropies $\varphi$ such that $z \varphi^{''} \lesssim  \varphi^{'},$ typically, any power-like entropy.
\end{remark}

\begin{remark}
The restriction $\varepsilon \le \nu$ is mandatory to enable the
choice of a parameter $\lambda$ such that the Korteweg term
disappears in the system for $(R,W_\lambda)$, see \eqref{eq:RWlambda-eps}.  
\end{remark}

\subsection{Compactness of weak solutions} 

%This section is devoted to the proof of Theorem~\ref{theo:compacite}:
In this section we assume $\nu >0$ and $0 \le \varepsilon \le \nu$, and we consider the isothermal case 
$P(\rho)\equiv \kappa \rho$ with $\kappa >0$.
From Section \ref{sec:apriori}, any classical solution $(R,U)$ to \eqref{eq:RU} on $(0,T)$ decaying sufficiently fast at infinity satisfies the following {\em a priori} estimates: 

\begin{itemize}
\item[$\bullet$] The conservation of mass:
\begin{equation} \label{eq_mass}
\sup_{t \in (0,T)} \int_{\R^d} R = M_0 
\end{equation}
where $M_0$ is the mass of the initial data,
\item[$\bullet$] From the dissipation of the pseudo energy:
\begin{multline} \label{eq_dissipation_eps<nu}
 \sup_{t \in (0,T)} \left\{ \dfrac{1}{2\tau^2} \int_{\R^d} \left( R|U|^2 + \eps^2|\nabla \sqrt R|^2 \right) + \kappa \int_{\R^d} R ( |y|^2 +  |\ln R|) \right\}  \\
+ \int_0^T \left( \dfrac{\dot{\tau}}{\tau^3} \int_{\R^d} \left( R
    |U|^2 + \eps^2|\nabla \sqrt R|^2 \right)   + \dfrac{\nu}{\tau^4}
  \int_{\R^d} |\mathbb S_{N}|^2 \right) \dd t 
\le  C_0 , 
\end{multline}
where $C_0$ depends on initial data only. 
\item[$\bullet$] From the dissipation of the pseudo BD-entropy:
\begin{multline}\label{eq_BD_eps<nu}
\sup_{t \in (0,T)} \left\{\dfrac{1}{2\tau^2} \int \left( R |U + \nu \nabla \ln R|^2 + \eps^2 |\nabla \sqrt R|^2 \right)  + \kappa\int R (|y|^2 + |\ln R|) \right\} \\
+ \int_{0}^T \left(\dfrac{\dot{\tau}}{\tau^3} \int\left( R |U|^2 +
    \eps^2|\nabla \sqrt R|^2 \right)   
+ \dfrac{\nu}{\tau^4} \int |\mathbb A_N|^2 
+ \dfrac{4\nu \kappa }{\tau^2} \int |\nabla \sqrt{R}|^2 
 \right)  \dd t \\
+ \int_{0}^T \dfrac{\nu \eps^2 }{4\tau^4} 
\left(   \int R\left|\nabla^2\ln R\right|^2\right)  \dd t
\le C'_0 , 
\end{multline} 
 where $C'_0$ depends again only on initial data and $\mathbb A_N$
 stands for the skew-symmetric part of $\mathbb
 T_N$. 
\item[$\bullet$] The Mellet-Vasseur type inequality: denoting $\varphi_{MV}(z) = (1+z) \ln(1+z)$, there holds
\begin{equation} \label{eq_MV_eps<nu}
  \begin{aligned}
    \sup_{t \in (0,T)} &\left\{ \int R \varphi_{MV}\left(|W_\eps|^2 +
        |y|^2 \right) \right\}  \\
&\quad +  \int_{0}^T \int  \varphi'_{MV}\left(|W_\eps|^2+
  |y|^2\right) \{  \lambda(\eps)|\mathbb T_N|^2 + \lambda_2(\eps)|\mathbb S_N|^2 \}  
\le  C''_{0,T} , 
  \end{aligned}
\end{equation} 
with $C''_{0,T}$ depending only on initial data and $T$, where hereafter in this section we define the effective velocity $W_\eps$ associated to $(R,U)$ by
\begin{equation*}
W_\eps := U + \lambda(\eps) \nabla \ln R, 
\end{equation*}
where we recall that $\lambda(\eps) = \frac{\nu - \sqrt{\nu^2 - \eps^2}}{2} \ge 0$ and $\lambda_2(\eps) = \sqrt{\nu^2 - \eps^2} \ge 0$.
\end{itemize}

We proceed by studying the compactness of weak solutions to
\eqref{eq:RU} which satisfy the estimates 
\eqref{eq_mass}--\eqref{eq_dissipation_eps<nu}--\eqref{eq_BD_eps<nu}--\eqref{eq_MV_eps<nu}. 
The different arguments follow closely the proof of
\cite[Theorem~2.1]{MelletVasseur}.    

%\subsection{Further study : compactness of sequences of weak solutions.}
%
%A standard question in the study of compressible fluids is the
%stability of solutions, closely related to the construction of weak
%solutions; see e.g. \cite{Lio98,Fei04}. 
%In this part, we consider the isothermal case, $P(\rho)\equiv
%\kappa \rho$. The proof of Theorem~\ref{theo:compacite} below can
%easily be adapted to the case 
%\begin{equation*}
%  P(\rho)=\kappa \rho +\sum_{j=1}^N
%  \kappa_j\rho^{\gamma_j},\quad N\ge 1,\ \kappa_j>0,
%\end{equation*}
%with the extra assumptions on $\gamma_j$, for $d\le 3$,
%\begin{equation*}
%  \gamma_j>1 \text{ if } d\le 2,\quad 1<\gamma_j< 3 \text{ if } d=3,
%\end{equation*}
%by proceeding as in \cite{MelletVasseur}.
%Since the technical
%modification can be found in \cite{MelletVasseur}, we choose to limit
%the technicalities by considering the isothermal case. 
%However, our argument does not seem to easily export to the more
%general Assumption~\ref{hyp:P-temps-long}. 

\begin{theorem}\label{theo:compacite}
Assume $\nu >0$ and $0 \le \varepsilon \le \nu$, $P(\rho)=\kappa \rho$ with $\kappa >0$, and
let $T>0$. 
Consider $(\sqrt{R_n}, \sqrt{R_n} U_n)_{n\in \mathbb N}$ a sequence of
weak solutions to \eqref{eq:RU} 
%associated to $(\sqrt{R_{n}^{\rm{in}}}
%, \sqrt{ R_{n}^{\rm{in}}} U_{n}^{\rm{in}})_{n \in \N}$ and 
satisfying
\eqref{eq_mass}-\eqref{eq_dissipation_eps<nu}-\eqref{eq_BD_eps<nu}-\eqref{eq_MV_eps<nu} 
with constants $C_0,C'_0,C^{''}_{0,T}$ independent of $n \in \mathbb
N$,  and denote by $\mathbb S_{K,n}$ and 
$\mathbb T_{N,n}$ the tensors associated to $(\sqrt{R_n}, \sqrt{R_n}
U_n)$. 
Then, there exists $(\sqrt{R},\sqrt{R}U)$, with associated tensors $\mathbb S_K$ and $\mathbb T_N$, such that:
\begin{itemize}
\item[i)] Up to the extraction of a subsequence, $(\sqrt{R_n},\sqrt{R_n} U_n,\mathbb T_{N,n})_{n\in \mathbb N}$ satisfy
\begin{align*}
& \sqrt{R}_n \to \sqrt{R} &&  \text{ in } C([0,T) ; L^2(\mathbb R^d) ) ,   \\
& \sqrt{R}_n U_n \to \sqrt{R}U &&\text{ in $L^2(0,T;L^2(\mathbb R^d)),$} \\
%& \mathbb T_{K,n} \rightharpoonup \mathbb T_K && \text{ in $L^2(0,T;L^1(\mathbb R^d))-w$} \\
& \mathbb T_{N,n} \rightharpoonup \mathbb T_N && \text{ in $L^2(0,T;L^2(\mathbb R^d))-w$},
\end{align*}
\item[ii)] $(\sqrt{R},\sqrt{R} U)$ is a weak solution to \eqref{eq:RU}
  in the sense of Definition~\ref{def_weaksol}. 
\end{itemize}
\end{theorem}

\begin{proof}[Proof of Theorem \ref{theo:compacite}]
To start with, we remark that, thanks to
\eqref{eq_mass}--\eqref{eq_dissipation_eps<nu}--\eqref{eq_BD_eps<nu},
the sequence we consider is bounded in the following respective spaces: 
\begin{itemize}
\item[({\bf b1})] 
$(\sqrt{R_n})_n$ is bounded in $L^{\infty}(0,T;H^1(\mathbb R^d) \cap
L^2 (\mathbb R^d ; |y|^2 \dd y)),$

\item[({\bf b2})] $(\sqrt{R_n}U_n)_n$ is bounded in
  $L^{\infty}(0,T;L^2(\mathbb R^d)),$ 

\item[({\bf b3})] $(\mathbb T_{N,n})_n$ is bounded in
  $L^2(0,T;L^2(\mathbb R^d)).$

\item[({\bf b4})] $(\nabla^2 \sqrt{R_n})_n$ is bounded in
  $L^2(0,T;L^2(\mathbb R^d))$ if $\eps >0$.

\end{itemize}
Up to the extraction of a subsequence, we can then construct
$(\sqrt{R},\sqrt{R}U,\mathbb T_N)$ as the following limits: 
\begin{itemize}
\item[({\bf c1})] $\sqrt{R_n} \rightharpoonup \sqrt{R}$ in $L^{\infty}(0,T;H^1(\mathbb R^d))-w*,$
\item[({\bf c2})] $\sqrt{R_n}U_n \rightharpoonup \sqrt{R}U$ in $L^{\infty}(0,T;L^2(\mathbb R^d)))-w*,$

\item[({\bf c3})] $\mathbb T_{N,n} \rightharpoonup \mathbb T_{N} $ in $L^2(0,T;L^2(\mathbb R^d))-w,$

\item[({\bf c4})] $\nabla^2 \sqrt{R_n} \rightharpoonup \nabla^2 \sqrt{R} $ in $L^2(0,T;L^2(\mathbb R^d))-w$ if $\eps >0$.

\end{itemize}
We note directly that the non-negativity of $\sqrt{R}$ is preserved in the weak limit. 

\medskip 

\paragraph{\em Step 1.}
From ({\bf b1}) and Sobolev embeddings, we have  
\begin{equation*} \label{eq_boundprecise} 
\tag{\textbf{b5}}
(\sqrt{R}_n)_{n} \text{ is bounded in } L^{\infty}(0,T;L^{q}(\mathbb R^d)) 
\text{ for all }
\begin{cases}
q \in [2,\infty) \text{ if } d = 2 , \\
q \in [2,2^*] \text{ if } d \ge 3 , 
\end{cases}
\end{equation*}
where  $2^* = 2d/(d-2)$. Together with (\textbf{b2}), this implies
\begin{equation*} \label{eq_bound-RnUn} 
\tag{\textbf{b6}} 
({R}_n U_n)_{n} \text{ is bounded in }  L^{\infty}(0,T;L^{q}(\mathbb R^d)) 
\text{ for all }
\begin{cases}
q \in [1,\infty) \text{ if } d = 2 , \\
q \in [1,d'] \text{ if } d \ge 3 , 
\end{cases}
\end{equation*}
where $d'$ is the H\"older conjugate exponent of $d$.
Recalling the continuity equation satisfied by $\sqrt{R_n}$,  
\[
\d_t\sqrt{R}_n = -\frac{1}{\tau^2}\DIV ( \sqrt{R}_n U_n ) +
\frac{1}{2\tau^2}{\rm Trace}(\mathbb T_{N,n}), 
\]
and that $\tau$ is uniformly bounded from below on $(0,T)$, the bounds
(\textbf{b1})--(\textbf{b2})--(\textbf{b3}) yield that $(\d_t \sqrt
R_n)_n$ is bounded in $L^2(0,T ; H^{-1}(\R^d))$. Consequently, as in
\cite[Lemma 4.1]{MelletVasseur}, we apply Aubin-Lions' lemma in the
form \cite[Corollary 4]{Simon87} with the triplet $H^1(\mathbb R^d)
\cap L^2 (\mathbb R^d ; |y|^2 \dd y) \subset L^{2}(\R^d) \subset
H^{-1}(\R^d)$, where the first embedding is compact. This yields that   
\begin{equation*} \label{eq_strongconvergence_R} \tag{\textbf{c5}} 
\text{$(\sqrt{R}_n)_{n}$ is relatively compact in $C([0,T];L^{2}(\mathbb R^d)).$}
\end{equation*}
Furthermore, in the case $\eps >0$, estimates
(\textbf{b1})--(\textbf{b4}) imply that $(\sqrt{R_n})_n$ is bounded in
$L^{2}(0,T;H^2(\mathbb R^d))$ which yields, applying Aubin-Lions'
lemma again, that 
\begin{equation*} \label{eq_strongconvergence_R_eps} \tag{\textbf{c5'}} 
\text{$(\sqrt{R}_n)_{n}$ is relatively compact in $L^2(0,T;H^{1}_{\rm
    loc}(\mathbb R^d))$ if $\eps >0$. } 
\end{equation*}

\medskip 
\paragraph{\em Step 2.}
The second step of the proof is to obtain the relative compactness of
$(R_nU_n)_{n}.$ We remark that, by definition: 
\[
\nabla (R_nU_n) = \sqrt{R_n} \mathbb T_{N,n} + 2 \sqrt{R_n} U_n \otimes \nabla \sqrt{R_n}.
\]
We combine here (\textbf{b1})--(\textbf{b2})--(\textbf{b3}). This yields
that $(\nabla( R_n U_n))_n$ is bounded in $L^{2}(0,T;L^{1}(\mathbb
R^d))$, hence $(R_n U_n)_{n}$ is bounded in $L^{2}(0,T;W^{1,1}(\mathbb
R^d))$ thanks to (\textbf{b5}). 
As for $\partial_{t}(R_nU_n),$ we apply the momentum equation to write:
\[
\begin{aligned}
\partial_t (R_nU_n) 
&=   - \frac{1}{\tau^2}\DIV ( \sqrt{R_n}U_n \otimes \sqrt{R_n}U_n)
- 2\kappa y R_n -  \kappa \nabla R_n   
+ \dfrac{\nu \dot{\tau}}{\tau} \nabla R_n \\
&\quad 
+ \dfrac{\eps^2}{2 \tau^2} \DIV  \mathbb S_{K,n}  
+ \dfrac{\nu}{\tau^2} \DIV ( \sqrt{R_n}\mathbb S_{N,n} ),
\end{aligned}
\]
where we recall that $\mathbb S_{K,n} =  \sqrt{R_n} \nabla^2
\sqrt{R_n} -  \nabla\sqrt{R_n} \otimes \nabla \sqrt{R_n} $.  
Again, the bounds on $\sqrt{R_n},\sqrt{R_n}U_n$, $\mathbb T_{N,n}$ and
$\nabla^2  \sqrt{R_n}$ coming from
(\textbf{b1})--(\textbf{b2})--(\textbf{b3})--(\textbf{b4}) imply that
$\partial_t (R_n U_n)$ is bounded in $L^{2}(0,T;W^{-1,1}(\mathbb
R^d)).$ So, by the Aubin-Lions' lemma with the triplet 
$W^{1,1}(K) \subset L^{p}(K) \subset W^{-1,1}(K)$ for any $p \in [1,d')$ and any compact $K \subset \R^d$, where the first embedding is compact, we obtain that  
\begin{equation*} \label{eq_cvgmomentum} \tag{\textbf{c6}}
\text{$({R}_nU_n)_{n}$ is relatively compact in $L^2(0,T;L^{p}_{\rm loc}(\mathbb R^d))$ for all $p \in [1,d')$.} 
\end{equation*}

In what follows we assume that we have extracted a subsequence (that we do not relabel) so that we have the convergences:
\begin{itemize}
\item $\sqrt{R_n} \to \sqrt{R}$ in $C([0,T];L^2(\mathbb R^d))$;

\item $R_n U_n \to M$ in $L^2(0,T;L^p_{\rm loc}(\mathbb R^d))$ for any $1 \le p < d'$;

\item $\sqrt{R_n} \to \sqrt{R}$ in $L^2(0,T;H^{1}_{\rm loc}(\mathbb R^d))$ if $\eps > 0$.
\end{itemize}
We add here that \eqref{eq_boundprecise} entails that the sequence $(R_n)_{n}$ is bounded in 
 $L^{\infty}(0,T;L^{q/2}(\mathbb R^d))$ for any $2\le q < \infty$ if
 $d=2$ and any $2 \le q \le 2^*$ if $d \ge 3$, hence it admits (up to
 the extraction of a subsequence) a weak-$*$ limit. Thanks to the
 strong convergence ({\bf c5}) of $(\sqrt{R}_n)_n$ we have that  
\begin{equation*} \label{eq_cvgcdensity} \tag{\textbf{c7}} 
R_n \rightharpoonup  R \text{ in $L^{\infty}(0,T;L^{q/2}(\mathbb R^d))-w*$ for all }
\begin{cases}
q \in [2,\infty) \text{ if } d = 2 , \\
q \in [2,2^*] \text{ if } d \ge 3 .
\end{cases}
\end{equation*}

\medskip

\paragraph{\em Step 3.}
We proceed with defining the asymptotic velocity-field $U.$ For this,
we remark first that, for arbitrary $K\subset (0,T) \times \mathbb
R^d$ there holds, for arbitrary $ 2 < q <2^*$ and $p$ such that $1/p=
1/2 + 1/q \in (1 - 1/d,1):$  
$$
\|R_n U_n\|_{L^p(K)} \le \|\sqrt{R_n}\|_{L^{q}(K)} \|\sqrt{R_n} U_n\|_{L^2((0,T) \times \mathbb R^d)}. 
$$ 
Taking $K = \{\sqrt{R} =0\} \cap \( (0,T) \times B(0,A)\)$ for arbitrary $A
>0$ we apply that $\sqrt{R_n} \mathbf{1}_{K} \to \sqrt R \mathbf 1_K =
0$ in $L^p(K)$ (by ({\bf c5})), and is bounded in $L^{r}(K)$ for
arbitrary $r \in (q,2^*)$ (by ({\bf b5})). By interpolation, we
conclude that $\sqrt{R_n} \mathbf{1}_{K} \to 0$ in $L^q(K).$
Recalling that $\|\sqrt{R}_n U_n\|_{L^2((0,T) \times \mathbb R^d)}$
remains bounded and that $R_n U_n \mathbf 1_{K} \to M$ in $L^p(K),$ we
infer that $M = 0$ on $\{\sqrt{R}=0\}.$ 
So, we set 
\[
U=
\left\{
\begin{array}{ll}
0  & \text{ on }\{\sqrt{R} = 0\},\\[8pt]
\dfrac{M}{R} & \text{ on }\((0,T) \times \mathbb R^d\) \setminus \{\sqrt{R} = 0\}.
\end{array}
\right.
\]
We note here that by construction 
\[
U = \lim_{n \to \infty} \dfrac{R_n U_n}{R_n} = \lim_{n \to \infty}
\dfrac{\sqrt{R_n} U_n}{\sqrt{R_n}}   \quad \text{ a.e. on $\((0,T)
  \times \mathbb R^d\) \setminus \{\sqrt{R} = 0\}$}. 
\]

In a similar fashion we define the asymptotic effective velocity field $W_\eps$ in the case $\eps >0$. We observe first that
\[
R_n U_n + 2\lambda(\eps) \sqrt{R_n} \nabla \sqrt{R_n}
\to M + 2\lambda(\eps) \sqrt{R} \nabla \sqrt{R} =: \bar M_\eps 
\text{ a.e.\ on } (0,T)\times \R^d, 
\]
and we have $ \bar M_\eps = 0$ on $\{\sqrt{R}=0\}.$ Hence we set
\[
W_\eps =
\left\{
\begin{array}{ll}
0  & \text{ on }\{\sqrt{R} = 0\},\\[8pt]
\dfrac{M}{R} +2\lambda(\eps) \dfrac{ \sqrt{R} \nabla \sqrt{R} }{R}&
            \text{ on }\((0,T) \times \mathbb R^d \)\setminus \{\sqrt{R} = 0\}. 
\end{array}
\right.
\]

\medskip 

\paragraph{\em Step 4.}
The last important step is to prove the strong convergence of the
sequence $(\sqrt{R_n} U_n)_{n}$ in $L^2(0,T;L^2_{\rm loc}(\mathbb
R^d)).$ In order to do so, we work with the effective velocity $
W_{\eps,n} = U_n + \lambda(\eps) \nabla \ln R_n$ (which is just equal
to $U_n$ when $\eps = 0$). We first remark that we have a.e.\
convergence of 
${R_n}\varphi_{MV}(|y|^2 + |W_{\eps,n}|^2)$. Estimate~\eqref{eq_MV_eps<nu} with
Fatou's Lemma yield 
\[
\sup_{(0,T)} \int R \varphi_{MV}(|y|^2 + |W_\eps|^2) < \infty.
\]
We may now repeat the arguments of \cite[pp.~445-446]{MelletVasseur} and \cite{AS-comp}. Namely, we first fix $A,A' >0$, and remark that 
$$
\sqrt{R_n} W_{\eps,n} = \frac{R_n U_n}{\sqrt{R_n}} + \nabla \sqrt{R_n} 
\to \frac{M}{\sqrt{R}} + \nabla \sqrt{R} \quad
\text{a.e.\ on } \{ \sqrt R \neq 0  \}, 
$$
as well as $|\sqrt{R_n} W_{\eps,n} \mathbf{1}_{|W_{\eps,n}| < A }  |
\le A \sqrt{R_n} \to 0 $ a.e.\ on $\{ \sqrt R = 0  \}$. Hence we get 

\[
\sqrt{R_n} W_{\eps,n} \mathbf{1}_{|W_{\eps,n}| < A \, \cap \,  |R_n| < A'}   \rightarrow 
\sqrt{R} W_{\eps} \mathbf{1}_{|W_{\eps}| < A \, \cap \,  |R| < A'}   
\text{ a.e. on $(0,T) \times\mathbb R^d$}.
\]
We write, for any compact $K \subset \R^d$, 
\[
\int_{0}^T \int_K |\sqrt{R_n} U_n - \sqrt{R} U |^2 
\le  C\int_{0}^T \int_K |\sqrt{R_n} W_{\eps,n} - \sqrt{R} W_{\eps} |^2 
+ C \lambda^2(\eps)\int_{0}^T \int_K |\nabla \sqrt{R_n}  - \nabla \sqrt{R} |^2
\]
and evaluate each term separately. The second term goes to $0$ thanks
to \eqref{eq_strongconvergence_R_eps}, while for the first term we
estimate 
\begin{equation}\label{eq:conv-RW}
\begin{aligned}
& \int_{0}^T \int_K |\sqrt{R_n} W_{\eps,n} - \sqrt{R} W_{\eps} |^2 \\
&
\le C \Bigg (  
 \int_{0}^T \int_K |\sqrt{R_n} W_{\eps,n} \mathbf{1}_{|W_{\eps,n}| < A \, \cap \,  |R_n| < A'}  - \sqrt{R} W_{\eps} \mathbf{1}_{|W_{\eps}| < A \, \cap \,  |R| < A'}  |^2 \\
&\qquad \qquad
+   \int_{0}^T \int_K  |\sqrt{R_n} W_{\eps,n} \mathbf{1}_{|W_{\eps,n}| \ge A}|^2 + |\sqrt{R_n} W_{\eps,n} \mathbf{1}_{|W_{\eps,n}| < A \cap |R_n |\ge A' }|^2 \\
&\qquad\qquad
 + \int_{0}^T \int_K  |\sqrt{R} W_{\eps} \mathbf{1}_{|W_{\eps}| \ge A}|^2 + |\sqrt{R} W_{\eps} \mathbf{1}_{|W_{\eps}| < A \cap |R| > A' }|^2
\Bigg).
\end{aligned}
\end{equation}
For fixed $A$ and $A'$, the first term on the right-hand side of \eqref{eq:conv-RW} converges to $0$ when $n\to \infty$, while for the second one we have, introducing $2 < q < 2^{*}$:
\begin{multline*}
\int_{0}^T \int_K  |\sqrt{R_n} U_n \mathbf{1}_{|U_n| \ge A}|^2 + |\sqrt{R_n} U_n \mathbf{1}_{|U_n| < A \cap |R_n| \ge A' }|^2\\
\begin{array}{l}
 \le \dfrac{1}{\ln(1+A^2)}\displaystyle \int_{0}^T \int_K R_n \varphi_{MV}((|y|^2 + |U_n|^2)/2 )  + \dfrac{A^2}{|A'|^{q-2}}\int_{0}^T \int_K R_n^{q}  \\[14pt]
 \le C \left( \dfrac{1}{\ln(1+A^2)} + \dfrac{A^2}{|A'|^{2-q}} \right),
\end{array}
 \end{multline*}
with a constant $C$ independent of $n.$ 
Proceeding in a similar way for the third term of \eqref{eq:conv-RW}, 
we obtain that 
\[
\limsup_{n \to \infty} \int_{0}^T \int_K |\sqrt{R_n} U_n - \sqrt{R} U |^2  
\le C \left( \dfrac{1}{\ln(1+A^2)} + \dfrac{A^2}{|A'|^{2-q}} \right),
\]
for arbitrary $A$ and $A'$, which implies the convergence 
\begin{equation} \label{eq_cvgcflux} \tag{\textbf{c8}}
\sqrt{R_n} U_n  \to \sqrt{R}U  \text{ in $L^2_{\rm loc}((0,T) \times \mathbb R^d)$}.
\end{equation} 
by letting $A' \to \infty$ and then $A \to \infty.$ We note here that, by construction 
$\sqrt{R} U = 0$ where $U =0$ in particular on the set $\{\sqrt{R}=0\}.$

We may finally combine (\textbf{c1})--(\textbf{c2})--(\textbf{c3})--(\textbf{c4})--(\textbf{c5})--(\textbf{c6})--(\textbf{c7})--(\textbf{c8}) to pass to the limit in the continuity and momentum equations satisfied by 
$(\sqrt{R_n},\sqrt{R_n}U_n)$ and their associated tensors $\mathbb S_{K,n}, \mathbb T_{N,n}$, and obtain that the different items of Definition~\ref{def_weaksol} are satisfied by the limit $(\sqrt{R},\sqrt{R}U)$ and their associated tensors $\mathbb S_{K}, \mathbb T_{N}$.
\end{proof}

%%%%%%%%%%%%%%%%%%%%%%%%%%%%%%%%%%%%%%%%%%%%%%

\appendix

\section{On large time behavior for isentropic Euler equations}
\label{sec:serre}

In this appendix, we prove Theorem~\ref{theo:serre}: for the Euler
equation with pressure 
law $P(\rho)=\rho^{\gamma}$, $\gamma>1$, there is no such thing as a
universal asymptotic profile for the density. In addition, the
dispersion associated to global smooth solutions is not the same as in
the isothermal case. To see this, we rewrite the arguments from
\cite{Serre97}, in the simplest case in order to illustrate the above
claims. Consider on $\R^d$, $d\ge 1$,
\begin{equation}\label{eq:euler-isent0}
  \left\{
    \begin{aligned}
 & \d_t \rho+\DIV\(\rho u\)=0,\\
     &\d_t (\rho u) +\DIV(\rho u\otimes u)+\kappa\nabla
        \(\rho^\gamma\)= 0,
    \end{aligned}
\right.
\end{equation}
with $\kappa>0$ and $1<\gamma\le 1+\frac{2}{d}$. Consider the analogue
of \eqref{eq:uvFluid},
\begin{equation*}
  \rho(t,x) = \frac{1}{(1+t)^d}R\(\frac{t}{1+t},\frac{x}{1+t}\),\quad
  u(t,x) = \frac{1}{1+t}U\(\frac{t}{1+t},\frac{x}{1+t}\)
  +\frac{x}{1+t}. 
\end{equation*}
Denoting by $\si$ and $y$ the time and space variables for $(R,U)$, we
readily check that in terms of $(R,U)$, \eqref{eq:euler-isent0} is 
equivalent to
\begin{equation}\label{eq:euler-isent2}
   \left\{
    \begin{aligned}
 & \d_\si R+\DIV\(RU\)=0,\\
     &\d_\si (RU) +\DIV( R U\otimes U)+\kappa(1-\si)^{d\gamma-d-2}\nabla
        \(R^\gamma\)= 0.
    \end{aligned}
\right.
\end{equation}
Note that in the case $\gamma=1+2/d$, $(R,U)$ solves exactly
\eqref{eq:euler-isent0}. This algebraic identity can be viewed as the
counterpart of the pseudo-conformal transform in the framework of
nonlinear Schr\"odinger equations (see e.g. \cite{CazCourant}), after
Madelung transform and a semi-classical limit (see
e.g. \cite{AnMa09,AnMa12,CaDaSa12}). (Leaving out the semi-classical limit, this
shows that at least in the case $\gamma=1+2/d$,
\eqref{eq:euler-isent0} could be replaced by Korteweg equations, with
essentially the same conclusions as below.)
\smallbreak
The important remark is that the time
interval $t\in [0,\infty)$ has been compactified, since it corresponds
to $\si\in [0,1)$. Therefore, if the solution of \eqref{eq:euler-isent2}
is defined (at least) on the time interval $[0,1]$, going back to the
original unknowns yields a global solution to
\eqref{eq:euler-isent0}. 
\smallbreak

We rewrite \eqref{eq:euler-isent2} away from vacuum as:
\begin{equation}\label{eq:euler-isent3}
   \left\{
    \begin{aligned}
 & \d_\si R+\DIV\(RU\)=0,\\
     &\d_\si U +U\cdot\nabla U+\kappa(1-\si)^{d\gamma-d-2}\frac{1}{R}\nabla
        \(R^\gamma\)= 0.
    \end{aligned}
\right.
\end{equation}
Using the same change of unknown function used
to symmetrize \eqref{eq:euler-isent0} (\cite{MUK86,JYC90}), but in the
case of $(R,U)$, that is,
\begin{equation*}
  \tilde R = R^{\frac{\gamma-1}{2}},
\end{equation*}
\eqref{eq:euler-isent3} becomes
\begin{equation}\label{eq:euler-isent4}
   \left\{
    \begin{aligned}
 & \d_\si \tilde R+U\cdot \nabla \tilde R +\frac{\gamma-1}{2}\tilde R
 \DIV U=0,\\
     &\d_\si U+U\cdot\nabla U+\kappa
     \frac{2\gamma}{\gamma-1}(1-\si)^{d\gamma-d-2}\tilde R \nabla 
        \tilde R= 0.
    \end{aligned}
\right.
\end{equation}
Multiplying the second equation by the symmetric positive definite matrix
\begin{equation*}
  S(\si) =
  \frac{(\gamma-1)^2}{4\kappa\gamma}\(1-\si\)^{d+2-d\gamma}{\mathrm
    I}_d 
\end{equation*}
makes the system symmetric. 
\smallbreak

\noindent {\bf Case $\gamma=1+2/d$.}
In this case, the symmetrizer is constant. 
Using the standard results in this
framework (see e.g. \cite{Majda,Taylor3}), we infer that if for $s>d/2+1$,
$\|(\tilde R,U)\|_{H^s(\R^d)}$ is sufficiently small at $\si=0$, then
\eqref{eq:euler-isent4} has a unique solution $(\tilde R,U)\in
C([0,1];H^s(\R^d))$. By the same argument, we can actually solve
\eqref{eq:euler-isent4}  backward in time, by prescribing the data at
$\si=1$: if these data are sufficiently small, the solution satisfies $(\tilde R,U)\in
C([0,1];H^s(\R^d))$. Back to the initial unknowns, we infer
Theorem~\ref{theo:serre}. 

\smallbreak

\noindent {\bf Case $1<\gamma<1+2/d$.} 
In this case, the symmetrizer $S$ goes to zero as $\si\to 1$. Setting,
for $m>1+d/2$ an integer
\begin{equation*}
 F_m(\si) :=\sum_{0\le |\alpha|\le m}\(\|\d_y^\alpha \tilde R\|_{L^2}^2 +
  \<\d_y^\alpha U,S\d_y^\alpha U\>_{L^2,L^2}\),
\end{equation*}
it is proven in \cite{Serre97} that $F_m$ satisfies the differential
inequality
\begin{equation*}
  \frac{dF_m}{d\si}\le C F_m+C
  \(1-\sigma\)^{d\gamma/2-1-d/2}\(F_m^{3/2}+F_m^{(m+3)/2}\) .
\end{equation*}
Defining $G_m(\si)=F_m(\si)\exp(-C\si)$, we get 
  \begin{equation*}
    \frac{d G_m}{d\si}\lesssim
    \(1-\sigma\)^{d\gamma/2-1-d/2}\(G_m^{3/2}+G_m^{(m+3)/2}\) .
  \end{equation*}
Introducing
\begin{equation*}
  H(G):= \int_1^G \frac{\dd g}{g^{3/2}+g^{(m+3)/2}} ,
\end{equation*}
we have 
\begin{equation*}
  H\(G_m(\si)\) \le H\(G_m(0)\) +C_1 \int_0^\si
  \(1-s\)^{d\gamma/2-1-d/2}\dd s.
\end{equation*}
Since the last integral is convergent as $\si\to 1$ (recall that
$\gamma>1$), and
\begin{equation*}
  H\(G_m(\si)\) \le H\(G_m(0)\) +\underbrace{C_1 \int_0^1
    \(1-s\)^{d\gamma/2-1-d/2}\dd s}_{:=C_2}.
\end{equation*}
Noticing that $H(0)=-\infty$, we see that if $\|(\tilde R,U)\|_{H^m}$
is sufficiently small, then $(\tilde R,U)$ is defined up to $\si=1$
(by contradiction). Again, we can adapt this argument with
data at $\si=1$ (replace $G_m(0)$ with $G_m(1)$ in the above
estimate), and decrease time to $\si=0$, in order to infer 
Theorem~\ref{theo:serre}. Note that in starting from $\si=1$, we only assume
$(\tilde R,U)_{\mid \si=1}$ in $H^m$, with $\|\tilde R(1)\|_{H^m}$
small (not necessarily $\|U(1)\|_{H^m}$).

%%%%%%%%%%%%%%%%%%%%%%%%%%%%%%%%%%%%%%%%%%%%%%

\section{Qualitative study of ordinary differential equations}
\label{sec:ODE}

\subsection{Universal dispersion}

We sketch the proof of Lemma~\ref{lem:taubis}, and refer to
\cite{CaGa-p} for details. The fact that under the assumptions of
Lemma~\ref{lem:taubis}, \eqref{eq:taugen} has a unique local $C^2$
solution is
an immediate consequence of Cauchy-Lipschitz Theorem. Multiplying
\eqref{eq:taugen} by $\dot \tau$ and integrating, we find
\begin{equation}\label{eq:taupoint}
  \(\dot \tau\)^2 = C+4\kappa \ln \tau,
\end{equation}
where the value
$  C= \beta^2 -4\kappa \ln \alpha$
is irrelevant for the rest of the discussion. Since the left hand
side of \eqref{eq:taupoint} is non-negative, we readily have
\begin{equation*}
  \tau(t)\ge \exp\(-\frac{C}{4\kappa}\)>0,
\end{equation*}
for all $t$ in the life-span of $\tau$. This shows that the $C^2$
solution is uniformly convex, and global in time. 
\smallbreak

Next, we note that $\tau$ grows at least linearly in time. Indeed, if
$\beta>0$, then since $\tau$ is convex, 
\begin{equation*}
  \tau(t)\ge \beta t+\alpha.
\end{equation*}
On the other hand, if $\beta\le 0$, suppose that $\tau$ is bounded,
$\tau(t)\le M$. Then \eqref{eq:taugen} yields
\begin{equation*}
  \ddot \tau(t)\ge \frac{2\kappa}{M},
\end{equation*}
hence a contradiction for $t$ large enough. Therefore, we can find
$T>0$ such that $\tau(T)\ge 1$ and $\dot \tau(T)>0$, so arguing like
above,
\begin{equation*}
  \tau(t)\ge \dot\tau(T)(t-T)+\tau(T),\quad \text{and}\quad \dot
 \tau(t)>0\quad \forall t\ge T. 
\end{equation*}
From the above discussion, there exists $T\ge 0$ such that for $t\ge
T$, $\dot \tau(t)>0$, and so \eqref{eq:taupoint} yields
\begin{equation}\label{eq:taupoint2}
  \dot \tau(t) = \sqrt{C+4\kappa \ln \tau(t)},\quad t\ge T. 
\end{equation}
Separating the variables, we have
\begin{equation*}
  \frac{\dd \tau}{\sqrt{C+4\kappa \ln \tau}}=\dd t,
\end{equation*}
and the change of variable $\sigma= \sqrt{C+4\kappa \ln \tau}$ yields
\begin{equation*}
  \int \frac{\dd \tau}{\sqrt{C+4\kappa \ln \tau}} =
  \frac{1}{2\kappa}\int e^{(\sigma^2-C)/4\kappa}\dd\sigma.
\end{equation*}
The asymptotic expansion of Dawson function (see e.g. \cite{AbSt64})
yields, in the sense of diverging integrals,
\begin{equation*}
  \int e^{\sigma^2}\dd\sigma \sim \frac{1}{2\sigma}e^{\sigma^2}.
\end{equation*}
We get
\begin{equation*}
  \frac{\tau(t)}{\sqrt{C+4\kappa \ln \tau}}\Eq t \infty t,\quad
  \text{hence}\quad
    \frac{\tau(t)}{\sqrt{4\kappa \ln \tau}}\Eq t \infty t.
\end{equation*}
We see here that the initial data of $\tau$, appearing in  the numerical
value of $C$, are irrelevant for the leading order large time behavior
of $\tau$. We readily infer 
\begin{equation*}
  \tau(t)\Eq t \infty 2t\sqrt{\kappa \ln t},\quad \dot \tau(t) \Eq t
  \infty 2\sqrt{\kappa \ln t},
\end{equation*}
where the second relation stems from the first one and
\eqref{eq:taupoint2}. 
\subsection{Perturbed dynamics}

The proof of Lemma~\ref{lem:ODE2} resume several of the above
steps. Local existence follows again from the Cauchy-Lipschitz
Theorem. Leaving out the explicit dependence upon $\eps$ and $\nu$ in
the notation, and multiplying \eqref{eq:tau-quant} by $\dot\tau$,
integration now yields
\begin{equation}\label{eq:taupoint-quant}
  \(\dot \tau(t)\)^2 = C+4\kappa \ln \tau(t) -\frac{\eps^2}{2\tau(t)^2}
  -\nu\int_0^t\(\frac{\dot \tau(s)}{\tau(s)}\)^2\dd s.
\end{equation}
Writing
\begin{equation*}
  C+4\kappa \ln \tau(t) = \(\dot \tau(t)\)^2
  +\frac{\eps^2}{2\tau(t)^2}+ \nu\int_0^t\(\frac{\dot
    \tau(s)}{\tau(s)}\)^2\dd s\ge 0,
\end{equation*}
we still have $\tau(t)\ge e^{-C/4\kappa}>0$. 
\smallbreak

Now suppose that $\tau\in L^\infty(\R_+)$. Then
\eqref{eq:taupoint-quant} and the above property imply
\begin{equation*}
  \(\dot \tau(t)\)^2+\nu\int_0^t\(\frac{\dot
    \tau(s)}{\tau(s)}\)^2\dd s\in L^\infty(\R_+),
\end{equation*}
hence
\begin{equation*}
\(\dot \tau(t)\)^2+\frac{\nu}{\|\tau\|_{L^\infty}^2}\int_0^t\(\dot
    \tau(s)\)^2\dd s\in L^\infty(\R_+).
\end{equation*}
In particular, $\int_0^\infty (\dot \tau)^2<\infty$. 
Integrating by parts,
\begin{align*}
  \int_0^t\(\dot    \tau(s)\)^2\dd s&=\tau(t)\dot\tau(t) -
  \alpha\beta-\int_0^t\tau\ddot \tau = \tau(t)\dot\tau(t) -
  \alpha\beta-\int_0^t\(2\kappa + \frac{\eps^2}{\tau^2}
    -\nu\frac{\dot \tau}{\tau}\)\\
& = \tau(t)\dot\tau(t) -
  \alpha\beta-\int_0^t\(2\kappa + \frac{\eps^2}{\tau^2}\)
  +\nu\ln\(\frac{\tau(t)}{\alpha}\). 
\end{align*}
Since $\tau$ is bounded, we
infer
\begin{equation*}
  \tau(t)\dot\tau(t) \gtrsim t-1,
\end{equation*}
hence a contradiction. Therefore, there exists $t_n\to \infty$ such
that
\begin{equation*}
  \tau(t_n)\Tend n \infty \infty. 
\end{equation*}
Now we suppose that
\begin{equation}\label{eq:intDV}
  \int_0^\infty\(\frac{\dot   \tau(s)}{\tau(s)}\)^2\dd s=\infty. 
\end{equation}
Then \eqref{eq:taupoint-quant} implies
\begin{equation}\label{eq:DVbrutale}
  4\kappa\ln \tau(t) - \(\dot \tau(t)\)^2\Tend t \infty \infty. 
\end{equation}
Integrating by parts yields
\begin{align*}
  \int_{t_n}^t  \(\frac{\dot   \tau}{\tau}\)^2 & =
 \frac{\dot\tau}{\tau}\Big|_{t_n}^t  -  \int_{t_n}^t  \dot \tau \(
   \frac{\ddot \tau}{\tau^2} -  2\frac{(\dot  \tau)^2}{\tau^3}\)\\
&= \frac{\dot\tau}{\tau}\Big|_{t_n}^t +2 \int_{t_n}^t \(\frac{\dot
  \tau}{\tau}\)^3 -\int_{t_n}^t \(2\kappa \frac{\dot \tau}{\tau^3}
  +\eps^2\frac{\dot \tau}{\tau^5}-\nu \frac{(\dot \tau)^2}{\tau^4}\)\\
& =
  \frac{\dot\tau}{\tau}+\frac{\kappa}{\tau^2}-\frac{\eps^2}{4\tau^4}\Big|_{t_n}^t 
+2 \int_{t_n}^t \(\frac{\dot  \tau}{\tau}\)^3 +\nu  \int_{t_n}^t
  \frac{(\dot \tau)^2}{\tau^4}. 
\end{align*}
In view of \eqref{eq:DVbrutale}, the above three integrated terms are
bounded. We infer
\begin{equation*}
   \int_{t_n}^t  \(\frac{\dot   \tau}{\tau}\)^2 \le C + \(2\sup_{s\ge
     t_n}\left| \frac{\dot \tau(s)}{\tau(s)}\right| +\nu\sup_{s\ge
     t_n}\frac{1}{\tau(s)^2}\) \int_{t_n}^t  \(\frac{\dot
     \tau}{\tau}\)^2. 
\end{equation*}
Now \eqref{eq:DVbrutale} yields, for $t\ge t_n\gg 1$,
\begin{equation*}
   \int_{t_n}^t  \(\frac{\dot   \tau}{\tau}\)^2 \le C + \frac{1}{2} \int_{t_n}^t  \(\frac{\dot
     \tau}{\tau}\)^2. 
\end{equation*}
This provides a contradiction with \eqref{eq:intDV}.
We infer that $\tau$ is not bounded, and 
\begin{equation*}
 \int_0^\infty\(\frac{\dot   \tau(s)}{\tau(s)}\)^2\dd s<\infty.   
\end{equation*}
But \eqref{eq:taupoint-quant} shows that for any sequence of time
along which $\tau$ goes to infinity, $(\dot \tau)^2$ also goes to
infinity. Therefore, 
\begin{equation*}
  \tau(t)\Tend t \infty \infty\quad\text{and}\quad  \dot\tau(t)\Tend t \infty \infty.
\end{equation*}
For large time, \eqref{eq:taupoint-quant} becomes
\begin{equation*}
   \(\dot \tau(t)\)^2 \Eq t \infty 4\kappa \ln \tau(t), 
\end{equation*}
and we can resume the computation of the above subsection to infer
Lemma~\ref{lem:ODE2}.

\bibliographystyle{smfplain}

\bibliography{biblio}

\end{document}